\documentclass[a4paper]{amsart}
\usepackage{amscd,amsmath,amssymb,amsfonts,verbatim}
\usepackage[cmtip, all]{xy}
\usepackage{pgf,tikz}
\usepackage{mathrsfs}
\usetikzlibrary{arrows}

\newtheorem{thm}{Theorem}[subsection]
\newtheorem{prop}[thm]{Proposition}
\newtheorem{lem}[thm]{Lemma}
\newtheorem{cor}[thm]{Corollary}

\theoremstyle{definition}

\newtheorem{defn}[thm]{Definition}
\theoremstyle{remark}
\newtheorem{remk}[thm]{Remark}
\newtheorem{remks}[thm]{Remarks}

\newtheorem{exm}[thm]{Example}
\newtheorem{exms}[thm]{Examples}
\newtheorem{notat}[thm]{Notation}
\numberwithin{equation}{subsection}

{\hfill$\square$\end{defn}}
{\hfill$\square$\end{remk}}
{\hfill$\square$\end{remks}}
{\hfill$\square$\end{exm}}
{\hfill$\square$\end{exms}}
{\hfill$\square$\end{notat}}

\newcommand{\CH}{{\rm CH}}
\newcommand{\BGH}{\mathbf{C}\mathbf{H}}

\newcommand{\red}{{\rm red}}
\newcommand{\codim}{{\rm codim}}

\newcommand{\Spec}{{\rm Spec \,}}
\newcommand{\Spf}{{\rm Spf \,}}
\newcommand{\Sp}{{\rm Sp \,}}

\newcommand{\Sch}{{\operatorname{\mathbf{Sch}}}}

\renewcommand{\max}{{\operatorname{\rm max}}}

\newcommand{\ds}{{/\kern-3pt/}}

\newcommand{\Tor}{{\operatorname{Tor}}}

\newcommand{\un}{\underline}

\renewcommand{\dim}{\text{\rm dim}}

\newcommand{\tuborg}{\left\{\begin{array}{ll}}
\newcommand{\sluttuborg}{\end{array}\right.}

\newcommand{\dashedrightarrow}{\dashrightarrow}

\newcommand{\llp}{(\! (}
\newcommand{\rlp}{)\!)}
\renewcommand{\mod}{ {\rm \ mod \ } }

\begin{document}
\title[Motivic cohomology of fat points]{Motivic cohomology of fat points in Milnor range via formal and rigid geometries}
\author{Jinhyun Park}
\address{Department of Mathematical Sciences, KAIST, 291 Daehak-ro Yuseong-gu, Daejeon, 34141, Republic of Korea (South)}
\email{jinhyun@mathsci.kaist.ac.kr; jinhyun@kaist.edu}


\keywords{algebraic cycle, Chow group, singular scheme, formal scheme, Tate algebra, rigid geometry, motivic cohomology, Milnor $K$-theory, algebraic de Rham cohomology, de Rham-Witt form}

\subjclass[2020]{Primary 14C25; Secondary 19D45, 13F25, 14B20, 16W60}

\begin{abstract}
We present a formal scheme based cycle model for the motivic cohomology of the fat points defined by the truncated polynomial rings $k[t]/(t^m)$ with $m \geq 2$, in one variable over a field $k$. We compute their Milnor range cycle class groups when the field has sufficiently many elements. 

With some aids from rigid analytic geometry and the Gersten conjecture for the Milnor $K$-theory resolved by M. Kerz, we prove that the resulting cycle class groups are isomorphic to the Milnor $K$-groups of the truncated polynomial rings, generalizing a theorem of Nesterenko-Suslin and Totaro. 
\end{abstract}

\maketitle

\setcounter{tocdepth}{1}

\tableofcontents

\section{Introduction}
In this article, we propose a new cycle model through formal geometry for the motivic cohomology of the singular non-reduced scheme $Y:=\Spec (k_m)$, where $k_m:= k[t]/(t^m)$, and we compute its Milnor range for all base fields $k$ except a finite number of finite fields, using some ideas from rigid analytic geometry. The main theorem is summarized in Theorem \ref{thm:main intro} below. 

\subsection{Background and motivation}
The motivic cohomology of smooth $k$-schemes is given by higher Chow groups of S. Bloch \cite{Bloch HC}, but the higher Chow complexes fail to give correct models of motivic cohomology when the schemes have singularities. 

\medskip

For such non-smooth schemes, approaches via additive higher Chow groups (\cite{BE2}, \cite{P2}, \cite{R}) and cycles with modulus (\cite{BS}) were proposed as part of potential extensions. Various theoretical advances were made (e.g. \cite{KMSY1}, \cite{KMSY2}, \cite{KMSY3}, \cite{KP crys}) in this direction. All of them use the idea of ``modulus conditions" imposed on cycles.

Recently in \cite{PU Milnor}, the author and S. \"Unver had tried a very different approach for $Y=\Spec (k_m)$. Embedding $\Spec (k_m)$ into the henselian scheme $\Spec (k[[t]])$, the article \emph{ibid.}~considered certain cycles on $\Spec (k[[t]]) \times \square_k ^n$ subject to a few requirements, where $\square_k := \mathbb{P}^1 \setminus \{ 1 \}$. 
On the cycles, the \emph{ibid.}~puts a ``mod $t^m$ equivalence relation", and the homology group of the resulting complex in the Milnor range was identified with the Milnor $K$-group $K_n ^M (k_m)$ when $k$ is of characteristic $0$.

\medskip

The author had received a few feedbacks since then. One of them from Fumiharu Kato at a social meeting in Tokyo, Japan in 2019 was that, for future generalizations to general singular schemes, it may be better to use formal schemes. For instance, instead of the scheme $\Spec (k[[t]])$, the formal scheme $\widehat{X}=\Spf (k[[t]])$ is the formal neighborhood of the immersion $Y= \Spec (k_m) \hookrightarrow \mathbb{A}^1$, so Kato's suggestion seems to be naturally consistent with the philosophy of the construction of the algebraic de Rham cohomology in A. Grothendieck \cite{Grothendieck DR} and R. Hartshorne \cite{Hartshorne DR}.

Once one accepts to try this perspective via formal schemes, it becomes necessary to ask what the right notions of (higher Chow) cycles for formal schemes are. To have a higher Chow cycle-stye theory, for instance, one should take the fiber product $ \Spf (k[[t]]) \times_k \square_k ^n$ in the category of formal schemes, instead of the fibre product $\Spec (k[[t]]) \times_k \square_k ^n$ of schemes. Here, one needs to answer what requirements one should impose on the closed formal subschemes to define the right theory.

\medskip

The first nontrivial test for a potential theory is to see whether its homology in the Milnor range for $Y$ is the Milnor $K$-group $K_n ^M (k_m)$. One far harder test is to show that it extends to a functorial theory on the category of schemes of finite type over $k$, where on the smooth schemes, it coincides with the classical higher Chow theory.

\medskip

The objective of this article is to prove that the theory we build here passes the first test in the Milnor range. The above harder test of constructing a functorial theory is studied in the sequel \cite{Park general}. A related discussion on algebraic $K$-theory of singular schemes is studied in \cite{PP cnv}.

\subsection{The central result}

More specifically, we do the following in this article. Consider the formal scheme $\widehat{X}= \Spf (k[[t]])$, with the closed immersion $Y:= \Spec (k_m) \hookrightarrow \widehat{X}$. Consider the formal scheme $\widehat{X} \times_k \square_k ^n$. 

We will construct (Definition \ref{defn:defn Milnor}) a complex of abelian groups 
\begin{equation}\label{eqn:intro cx}
\cdots \to z^q (\widehat{X} \mod Y, n+1) \overset{\partial}{\to} z^q (\widehat{X} \mod Y, n)  \overset{\partial}{\to}  z^q (\widehat{X} \mod Y, n-1) \to \cdots,
\end{equation}
where $z^q (\widehat{X} \mod Y, n)$ is the quotient 
\begin{equation}\label{eqn:intro cycle mod}
z^q (\widehat{X} \mod Y, n) = \frac{ z^q (\widehat{X} , n)}{ \mathcal{M}^q (\widehat{X}, Y, n)}
\end{equation}
of the group $z^q (\widehat{X}, n)$ consisting of certain ``admissible cycles" of codimension $q$ on the formal scheme $\widehat{X} \times_k\square_k^n$ modulo a subgroup $\mathcal{M}^q (\widehat{X}, Y,n)$ defined using the embedding $Y \hookrightarrow \widehat{X}$. The constructions of the groups $z^q (\widehat{X}, n)$ and $\mathcal{M}^q (\widehat{X}, Y, n)$ are rapidly sketched in \S \ref{sec:intro cycles} below. They are defined in full in Definitions \ref{defn:HCG} and \ref{defn:mod Y equiv}. 
The homology ${\rm H}_n (z^q (\widehat{X} \mod Y, \bullet))$ of the complex \eqref{eqn:intro cx} is denoted by any one of 
$$
\BGH^q (k_m, n) =\BGH^q (Y, n) = \CH^q (\widehat{X} \mod Y, n)= {\rm H}_n (z^q (\widehat{X} \mod Y, \bullet)),
$$
where the first two are written in boldface to distinguish them from the usual higher Chow group \cite{Bloch HC}. The central result of the article is to show that, under a mild assumption on the field $k$, we have an isomorphism $K_n ^M (k_m) \simeq \BGH^q (k_m, n)$. The following (from Proposition \ref{prop:q>n} and Theorem \ref{thm:k_m main}) gives more details. This includes a generalization to $\Spec (k_m)$ of theorems of Nesterenko-Suslin \cite{NS} and B. Totaro \cite{Totaro} proven for $\Spec (k)$:

\begin{thm}\label{thm:main intro}
Let $m \geq 2, n\geq 1$ be integers, $k$ be a field, and $k_m= k[t]/(t^m)$. 
\begin{enumerate}
\item If $q >n$, then we have $\BGH^q (\Spec (k_m), n) = 0$.
\item If $q=n$, there is a homomorphism, called the graph map
\begin{equation}\label{eqn:gr k_m intro}
gr_{k_m}: K_n ^M (k_m) \to \BGH^n (\Spec (k_m), n),
\end{equation}
and it is an isomorphism of abelian groups when the cardinality $|k|$ is sufficiently large.
\end{enumerate}
\end{thm}

The phrase ``sufficiently large" in Theorem \ref{thm:main intro} means that the cardinality of the field $k$ is larger than a positive integer $M_n$ given by M. Kerz \cite[Proposition 10-(5), p.181]{Kerz finite} so that the Milnor $K$-theory $K_n ^M( k[[t]])$ and the improved Milnor $K$-theory $\widehat{K}_n ^M (k[[t]])$ of Gabber-Kerz coincide. For instance, when $n=1$, this holds for an arbitrary field $k$.

\subsection{The cycles}\label{sec:intro cycles}
Recall $\widehat{X}= \Spf (k[[t]])$ and $Y= \Spec (k_m)$.
We sketch the construction of $z^q (\widehat{X} \mod Y, n)$ of \eqref{eqn:intro cycle mod}. The formal scheme $\widehat{X}$ has the largest ideal of definition $\mathcal{I}_0$ given by $(t)$. The reduction by $\mathcal{I}_0$ is $\widehat{X}_{\red} = \Spec (k)$.

\medskip

\subsubsection{The two requirements}\label{sec:3 require}

The group $z^q (\widehat{X}, n)$ is (Definition \ref{defn:HCG}) the free abelian group generated by the integral closed formal subschemes $\mathfrak{Z} \subset \widehat{X} \times_k \square_k ^n$ of codimension $q$, satisfying two conditions (\textbf{GP}) and (\textbf{SF}).

The general position property (\textbf{GP}) requires that the intersections of $\mathfrak{Z}$ with all faces $\widehat{X} \times F$ are proper, where $F \subset \square^n_k$ are given by a finite set of equations of the form $\{ y_i = \epsilon_i\}$ with $\epsilon_i \in \{ 0, \infty \}$. The special fiber property (\textbf{SF}) requires that the intersections of $\mathfrak{Z}$ with all special fiber faces $\widehat{X}_{\red} \times F = \Spec (k) \times F$ are proper.

\subsubsection{The mod $Y$-equivalence}\label{sec:intro modulus}
We discuss the subgroup $\mathcal{M}^q (\widehat{X}, Y, n) \subset z^q (\widehat{X}, n)$ of \eqref{eqn:intro cycle mod}, which defines the mod $Y$-equivalence on the cycles. 

This subgroup is generated by the cycles of the form
$$
[\mathcal{A}_1]- [\mathcal{A}_2],
$$
 where $(\mathcal{A}_1, \mathcal{A}_2)$ runs over all possible pairs of coherent $\mathcal{O}_{\widehat{X} \times \square^n}$-algebras whose associated cycles belong to $z^n (\widehat{X}, n)$, such that there is an isomorphism 
 \begin{equation}\label{eqn:intro mod equiv}
 \mathcal{A}_1|_{Y \times \square^n} \simeq \mathcal{A}_2 |_{Y \times \square^n}
 \end{equation}
 as $\mathcal{O}_{Y \times \square^n}$-algebras (see Definition \ref{defn:mod Y equiv}).

 \medskip

\subsection{Benefits of the formal geometry model}\label{sec:1.4}

One point of using the formal scheme model in studying the motivic cohomology of $\Spec (k[t]/(t^m))$ is its natural generalizability to more general schemes $Y$ of finite type in $\Sch_k$. This aspect is studied in \cite{Park general} in terms of some ideas of the derived algebraic geometry of J. Lurie.

Yet, there is another aspect that we mention, compared to the scheme model in \cite{PU Milnor}. In \emph{ibid.}, the generating cycles in the Milnor range were integral closed subschemes in $\Spec (k[[t]]) \times \square^n_k$ subject to the three requirements: (i) the proper intersection with the faces, (ii) the proper intersection with the special fiber times faces, and (iii) the extra requirement that each integral cycle is \emph{proper} over $k[[t]]$, thus finite over $k[[t]]$.

The first two conditions (i) and (ii) correspond to our requirements (\textbf{GP}) and (\textbf{SF}), respectively. However, to exclude potentially undesirable cycles, the extra condition (iii) was necessary there. One can ask what happens to this condition in our formal scheme model. 

One interesting aspect we observe in the formal scheme model is that, in fact, in the Milnor range we \emph{automatically} have the finiteness of the generating integral cycles over $k[[t]]$. This surprising conclusion in the formal scheme model is deduced in \S \ref{sec:appendix 01}, especially in Corollary \ref{cor:finite quasi-finite}. This argument does not work in the scheme model, as we really exploit the situation in the formal geometry. 

In connection to this, in Example \ref{exm:TS bad empty}, we consider an integral closed cycle $\mathfrak{Z}_1 \subset \Spec (k[[t]]) \times \square_k ^2$ given by two concrete equations, such that it satisfies the two conditions (i), (ii) but not (iii); $\mathfrak{Z}_1$ is not finite over $k[[t]]$. It is forcibly removed in the scheme model by (iii).

In contrast, the corresponding formal scheme $\mathfrak{Z}_2 \subset \Spf (k[[t]]) \times \square_k ^2$ given by the same equations is empty in the formal scheme model, thanks to the convergence of a sequence in the $(t)$-adic topology. 

This shows the formal scheme model of this article has various stronger points over the scheme model, e.g. of \cite{PU Milnor}.

\subsection{Connections to rigid geometry}
We explain how and where the rigid analytic geometry enters into the landscape of algebraic cycles. 

Using the automorphism $y \mapsto y/ (y-1)$ of $\mathbb{P}^1$, identify $\square^1$ with $\mathbb{A}^1$. So for $y_i' := y_i/ (y_i -1)$ where $(y_1, \cdots, y_n) \in \square^n$, the ambient space $\widehat{X} \times \square_k ^n$ can be seen as the formal spectrum of the ring $k[[t]] \{ y_1' , \cdots, y_n'\}$ of the \emph{restricted formal power series}, i.e.~the formal power series $\sum_I a_I {y'}^I$ in $y_1', \cdots, y_n'$ such that the coefficients $a_I  \in k[[t]]$ converge to $0$ in the $(t)$-adic norm as the degree $|I| \to \infty$ (see Example \ref{exm}). 

Here, not all prime ideals of $k[[t]]\{y_1', \cdots, y_n '\}$ are suitable for our cycles. In (\textbf{SF}) in \S \ref{sec:3 require} or Definition \ref{defn:HCG}, we ruled out those behave badly with respect to the ideals of definition of $\widehat{X} \times_k \square_k^n$. Those ideals give the ``special fibers" $\Spec (k) \times_k F $ for faces $F \subset \square^n_k$, and the condition (\textbf{SF}) requires the proper intersection of the cycles with such special fibers. 

In particular, via localization to the ``generic fibers", our admissible prime ideals of height $n$ in the Milnor range define prime ideals of height $n$ of the Tate algebra $T_n = k\llp t \rlp \{ y_1', \cdots, y_n' \}$ over the non-archimedean complete discrete valued field $k\llp t \rlp $. 

The Tate algebras were introduced by J. Tate in \cite{Tate}. Since $T_n$ has the Krull dimension $n$, the height $n$ primes are its maximal ideals. They define finite dimensional Banach algebras over $k\llp t \rlp $. One can show that in $T_n$, all such maximal ideals are complete intersections given by $n$ \emph{polynomials} in $y_1', \cdots, y_n'$. Thus, taking closures, we deduce algebraic generators for the integral admissible cycles in $z^n (\widehat{X}, n)$. This is covered in \S \ref{sec:generic fiber}. This allows us to see in \S \ref{sec:appendix 01} that each integral cycle $\mathfrak{Z}$ can be written as $\Spf (R)$, where $R =  k[[t]] \{ y_1, \cdots, y_n \}/ P$ for some prime ideal, and it is finite over $k[[t]]$.

\medskip

Another place, where the rigid geometry plays a role, is in the middle of the proof that the graph map \eqref{eqn:gr k_m intro} is an isomorphism.

As a first (eventually wrong) attempt, one might try to ``construct" its inverse
\begin{equation}\label{eqn:intro inv gr}
\BGH^n (\Spec (k_m), n) \to K_n ^M (k_m),
\end{equation}
 as follows: let $\bar{y}_i$ be the image of $y_i$ in $R$. Then one may try to define
$$
\rho: z^n (\Spf (k[[t]]), n) \to K_n ^M (k[[t]])
$$
by sending an integral cycle $\mathfrak{Z}$ to the norm $N_{R/k[[t]]} \{ \bar{y}_1, \cdots, \bar{y}_n \} \in K_n ^M (k[[t]])$ of the Milnor symbol $\{ \bar{y}_1, \cdots, \bar{y}_n \} \in K_n ^M (R)$. Then one may try to show that its composition with the natural surjection $K_n ^M (k[[t]]) \to K_n ^M (k[t]/(t^m))$ kills the boundary $\partial z^n (\Spf (k[[t]]), n+1)$ as well as the subgroup $\mathcal{M}^n (\Spf (k[[t]]), Y, n)$.

Unfortunately, the above proposed story does not work well. The problem is that we \emph{do not} know whether there is the norm map $N_{R/ k[[t]]}: K_n ^M (R) \to K_n ^M (k[[t]])$ unless $k[[t]] \to R$ is finite \'etale (see M. Kerz \cite[Propositions 3, 10]{Kerz finite}). In general, its existence is unknown at present as far as the author is aware of.

\medskip

To get around this difficulty, we take a detour through the rigid geometry via the generic fiber of the cycles, and use the Gersten conjecture for Milnor $K$-theory proven by M. Kerz in \emph{ibid.}

We first (see \S \ref{sec:suslin}) construct the composite
$$
\widehat{\Upsilon}: z^n (\Spf (k[[t]]), n) \overset{\eta}{\to} \CH^n ({\rm Sp} (k\llp t\rlp), n) \overset{\tilde{\Upsilon}}{\to} K_n ^M (k\llp t\rlp  ),
$$
where the middle term is a rigid analytic analogue of the higher Chow group (see Definition \ref{defn:HCG rigid}), $\eta$ is the localization map, and $\tilde{\Upsilon}$ is constructed using the norm map of Bass-Tate \cite{BassTate} and Kato \cite{Kato} on the Milnor $K$-theory for \emph{fields}. We check that the image of $\widehat{\Upsilon}$ belongs to the \emph{subgroup} $K_n ^M (k[[t]])$ of $K_n ^M (k\llp t\rlp  )$, which requires the Gersten conjecture for $K_n ^M$ of M. Kerz \cite{Kerz finite}. The rest of the paper is, roughly speaking, about proving that this map descends mod $t^m$, and that $gr_{k_m}$ in \eqref{eqn:gr k_m intro} is an isomorphism.

\medskip

These offer an interesting new observation: the motivic cohomology of the truncated polynomials can be partially understood through the rigid analytic geometry. Already in \cite{PU Milnor}, the topology given by the non-archimedean $(t)$-adic norm was very convenient. The author hopes that the new door opened here could be useful for the researchers in the field.

\subsection{The relative parts}

Taking the relative parts of the groups of Theorem \ref{thm:main intro}, we deduce the following result, stated as Theorem \ref{thm:rulling}:

\begin{thm}\label{thm:main intro 2}
Let $m \geq 1, n \geq 1$ be integers. Let $k$ be a field of characteristic $0$.

Then we have an isomorphism from the group of the big de Rham-Witt forms
\begin{equation}\label{eqn:gr rel intro}
\mathbb{W}_m \Omega_k ^{n-1} \overset{\sim}{\to} \BGH^n ( (\Spec (k_{m+1}), (t)), n)
\end{equation}
to the relative group of the new higher Chow group. 
\end{thm}

While we guess that the statement of Theorem \ref{thm:main intro 2} may stay valid for any field $k$, at this moment the author was unable to give a proof of it without imposing the assumption on the characteristic. It should be removed in the future. 
The group of the big de Rham-Witt forms on the left of \eqref{eqn:gr rel intro} is defined by Hesselholt-Madsen \cite{HeMa}, while the relative group on the right of \eqref{eqn:gr rel intro} is defined to be (see \eqref{eqn:rel CH})
$$
\ker ( \BGH^n (\Spec (k_{m+1}), n) \overset{ev_0}{\to} \CH^n (k, n)),
$$
where $\CH^n (k, n)$ is the usual higher Chow group of S. Bloch \cite{Bloch HC}.

This Theorem \ref{thm:main intro 2} is an analogue of the theorem of K. R\"ulling \cite{R} proven in terms of additive higher Chow groups (see also \cite{KP crys} for its higher dimensional generalization). An interesting point is that our theory in this article provides a single unified platform for all of Nesterenko-Suslin \cite{NS}, Totaro \cite{Totaro} and R\"ulling \cite{R}, and it is also relatively conceptual compared to the approaches through ``cycles with modulus" in the literature.

\bigskip

\textbf{Conventions.}
In this paper $k$ is a given arbitrary field, unless said otherwise. For the set inclusion symbol $\subset$, this allows the equality $=$ as well. The strict inclusion is denoted by $\subsetneq$, while we won't use the symbol $\subseteq$.
 
 The notations $k_m$, $Y$, $X$, $\widehat{X}$ in this article exclusively mean the following: for an integer $m \geq 1$, we let $k_m:= k[t]/(t^m) = k[[t]]/(t^m)$, $Y= \Spec (k_m)$, $Y \hookrightarrow X=\mathbb{A}^1$ is the closed immersion given by $k[t] \to k_m$, and $\widehat{X}$ is the completion of $X$ along $Y$, so that $\widehat{X} = \Spf (k[[t]])$. \qed

\section{Some definitions and recollections}\label{sec:cycles}

In \S \ref{sec:cycles}, we discuss the notion of higher Chow cycles on some noetherian affine formal $k$-schemes of finite Krull dimension. Based on these, in \S \ref{sec:mod Y} we define the main objects of studies in the article, the groups $\BGH^q (k_m, n)$. In \S \ref{sec:Tate}, we also discuss basic definitions and results on Tate algebras, that will be used later to compute the Milnor range $\BGH^n (k_m, n)$, when $q=n$.

\subsection{Noetherian affine formal schemes and cycles}\label{sec:naive cycle}

Let $\mathfrak{X}= \Spf (A)$ be a noetherian affine formal scheme of finite Krull dimension. 

\begin{defn}
An \emph{integral closed formal subscheme} $\mathfrak{Y} \subset \mathfrak{X}$ is of the form $\Spf (A/P)$ for a prime ideal $P \subset A$. 

The \emph{naive group} of algebraic cycles  $\un{z}_* (\mathfrak{X})$ on $\mathfrak{X}$ is defined to be the free abelian group on such integral closed formal subschemes. This is equivalent to the group of algebraic cycles on the \emph{scheme} $\Spec (A)$, that we call the associated scheme of $\mathfrak{X}$. Consequently, we have the same notion of intersections of cycles, and proper intersections of them. 

Its subgroup of $d$-dimensional cycles is denoted by $\un{z}_d (\mathfrak{X})$. In case $\mathfrak{X}$ is equidimensional (i.e. $A$ is equidimensional), we denote the group of the codimension $q$-cycles by $\un{z}^q  (\mathfrak{X})$.

\medskip

Not all of the cycles in $\un{z}_* (\mathfrak{X}) = z_* (\Spec (A))$ are suitable for our discussions. A better one is the subgroup $z_* (\mathfrak{X}) \subset \un{z}_* (\mathfrak{X})$ of cycles that intersect properly with $\mathfrak{X}_{\red}$. This $\mathfrak{X}_{\red}$ is the closed subscheme of $\mathfrak{X}$ given by the largest ideal of definition, which exists by EGA I \cite[Proposition (10.5.4), p.187]{EGA1}.
\qed
\end{defn}

\medskip

When $\mathcal{F}$ is a coherent $\mathcal{O}_{\mathfrak{X}}$-module, we have the associated cycle $[\mathcal{F}] \in \un{z}_* (\mathfrak{X})$ via minimal associated primes (see e.g. \cite[Section 02QV]{stacks}). This is possible because $\mathcal{F}= M^{\Delta}$ for a finitely generated $A$-module $M$ (see EGA I \cite[Proposition (10.10.2), p.201]{EGA1}), and the associated cycle of $M$ on $\Spec (A)$ defines $[\mathcal{F}]$.

\subsection{Higher Chow cycles over affine formal schemes}\label{sec:HC}

Let $\square_k := \mathbb{P}_k ^1 \setminus \{ 1 \}$. Let $\square^0 _k := \Spec (k)$. For $n \geq 1$, let $\square^n_k$ be the $n$-fold fiber product of $\square_k$ over $k$. Let $(y_1, \cdots, y_n) \in \square^n_k$ be the coordinates.

Recall that in the cubical version of higher Chow theory, using the $k$-rational points $\{ 0, \infty \} \subset \square_k$, we define a face $F \subset \square_k ^n$ to be a closed subscheme given by a system of equations of the form $\{ y_{i_1} = \epsilon_1, \cdots, y_{i_s} = \epsilon_s\}$ for an increasing sequence $1 \leq i_1 < \cdots < i_s \leq n$ of indices and $\epsilon_j \in \{ 0, \infty \}$. When the set of equations is empty, we let $F= \square_k ^n$ by convention. For $s=1$, we have the codimension $1$ faces given by $\{ y_i = \epsilon \}$ for some $1 \leq i \leq n$ and $\epsilon \in \{ 0, \infty \}$.

For an equidimensional noetherian affine formal $k$-scheme $\mathfrak{X}$ of finite Krull dimension, consider the fiber product $\mathfrak{X} \times_k \square_k ^n$, which we often write $\square_{\mathfrak{X} } ^n$. The fiber product exists in the category of formal schemes; see EGA I \cite[Proposition (10.7.3), p.193]{EGA1}.
A face $F_{\mathfrak{X} }$ of $\square_{\mathfrak{X} } ^n$ is given by $\mathfrak{X}  \times_k F$ for a face $F \subset \square_k ^n$.

\begin{exm}\label{exm}
Suppose $\mathfrak{X}= \Spf (A)$. Using the automorphisms of $\mathbb{P}^1$ given by $y_i \mapsto y_i':= y_i / (y_i -1)$ for $1 \leq i \leq n$, we may identify $\square_{\mathfrak{X}} ^n$ with the integral noetherian affine formal $k$-scheme $\mathfrak{X} \times_k \mathbb{A}_k ^n$ given by the ring $ A \{ y_1', \cdots, y_n' \}$ of restricted formal power series.

In case $A= k[[t]]$ with the $(t)$-adic topology, this ring is 
$$
k[[t]] \{ y_1' , \cdots, y_n'  \} =\underset{m}{ \varprojlim} \ k_m [ y_1 ' , \cdots, y_n ' ].
$$
This ring can also be regarded as the subring of $k[[t]][[y_1 ' , \cdots, y_n ' ]]$, whose members are formal power series 
$$
p(y_1' , \cdots, y_n' ) =\sum_{I} \alpha_I {y'}^I
$$
with $\alpha_I \in k[[t]]$ for multi-indices $I= (i_1, \cdots, i_n)$ with $i_j \geq 0$, such that when $|I|=i_1 + \cdots + i_n \to \infty$, we have $\alpha_I \to 0$ in the $(t)$-adic topology.
\qed
\end{exm}

We define cubical higher Chow cycles over the formal scheme $\mathfrak{X}$ (\emph{cf.} \cite{Bloch HC}) as follows. 

\begin{defn}\label{defn:HCG}
Let $\mathfrak{X}$ be an equidimensional noetherian affine formal $k$-scheme of finite Krull dimension. Let $q, n \geq 0 $ be integers. Let $\mathfrak{X}_{\red}$ be the closed subscheme given by the largest ideal $\mathcal{I}_0$ of definition of $\mathfrak{X}$.

Let $\un{z}^q (\mathfrak{X}, n) \subset \un{z}^q (\square_{\mathfrak{X}} ^n)$ be the subgroup generated by the codimension $q$ integral closed formal subschemes $\mathfrak{Z} \subset \square^n_{\mathfrak{X}}$ subject to the following conditions:
\begin{enumerate}
\item [(\textbf{GP})] (General Position) For each face $F \subset \square_k ^n$, $\mathfrak{Z}$ intersects $\mathfrak{X} \times F$ properly.
\item [(\textbf{SF})] (Special Fiber) For each face $F\subset \square_k ^n$, we have
$$\codim _{\mathfrak{X}_{\red} \times F} \ ( \mathfrak{Z} \cap (\mathfrak{X}_{\red} \times F)) \geq q.$$
\end{enumerate}
A cycle in $\un{z}^q (\mathfrak{X}, n)$ will be called \emph{admissible}.\qed
\end{defn}

We have several remarks on the conditions we introduced in Definition \ref{defn:HCG}. 

\begin{remk}
One notes that in addition to the usual general position condition (the condition (\textbf{GP})) with respect to the faces for higher Chow cycles, we require the additional condition (\textbf{SF}), which is related to the ``special fiber" given by the largest ideal of definition of $\mathfrak{X}$. In case $\mathfrak{X}$ is a \emph{scheme}, regarded as a formal scheme, the requirement (\textbf{SF}) holds automatically, so that Definition \ref{defn:HCG} is compatible with the definition of the cubical version of the higher Chow cycles of \cite{Bloch HC}, e.g. in \cite{Totaro}. \qed
\end{remk}

\begin{remk}
Note that by the special fiber condition (\textbf{SF}), the cycles given by open prime ideals of $A\{y_1' , \cdots, y_n' \}$ are all excluded from $\un{z}^q (\mathfrak{X}, n)$, where $\mathfrak{X}= \Spf (A)$ and $y_i': = y_i / (y_i -1)$ for $1 \leq i \leq n$ as before. \qed
\end{remk}

We make one elementary observation in Lemma \ref{lem:TS} below. In an earlier version of the article, an analogue of this property was considered as part of the requirements in Definition \ref{defn:HCG}, called the topological support condition (\textbf{TS}). Joseph Ayoub as well as the referee pointed out that, in fact, each integral cycle satisfies it automatically, so that it is redundant.

\begin{lem}\label{lem:TS}
Let $\mathfrak{X}$ be a noetherian formal scheme and let $\mathfrak{Z} \subset \mathfrak{X}$ be an integral closed formal subscheme. Let $\mathcal{I}$ be an ideal of definition of $\mathfrak{X}$ and let $Z$ be the reduction of $\mathfrak{Z}$ modulo $\mathcal{I}$. Then we have the equality of the topological supports $|\mathfrak{Z} |= |Z|$.
\end{lem}
\begin{proof}
Let $\mathcal{J}\subset \mathcal{O}_{\mathfrak{X}}$ be the ideal sheaf of $\mathfrak{Z}$. The ideal sheaf of $Z$ is given by $\mathcal{I} + \mathcal{J}$. However an ideal of definition of $\mathfrak{Z}$ is also given by $\mathcal{I} + \mathcal{J}$.
\end{proof}

\begin{exm}\label{exm:TS bad empty}
When $\mathfrak{X}= \Spf (A)$ for an adic ring $A$, one may ask whether we can consider the scheme $\Spec (A) \times \square^n$ (that uses polynomials) instead of the formal scheme as we do here (that uses restricted formal power series).

Here is an example that shows why our formal scheme based model via $\Spf (k[[t]]) \times \square^n_k$ considered in this article is better than the scheme based model via $\Spec (k[[t]]) \times \square^n_k$, e.g. considered in \cite{PU Milnor}. The author thanks Joseph Ayoub and the referee for sharing their insights and encouraging him to try to write it down.

In the scheme model, we may have some undesirable cycles such as (here, $n=2$ and $(y_1, y_2) \in \square^2$)
$$
 \mathfrak{Z}_1:  \{ y_1 = 1+ t, \ \ y_2 = t \},
 $$
which is not finite over $k[[t]]$. This is part of the reason why in \cite{PU Milnor}, the extra finiteness requirement was imposed to exclude forcibly this kind of cycles.

On the other hand, in the formal scheme model of ours, the corresponding closed formal subscheme $\mathfrak{Z}_2 \subset \Spf (k[[t]]) \times \square_k ^2$ given by the same defining equations,
$$
\mathfrak{Z}_2: \{ y_1 = 1+t, \ \ y_2 = t \},
$$
 is empty, thus it vanishes itself.

To see this, applying the automorphism $y \mapsto y/ (y-1)$ of $\mathbb{P}^1$ that sends $(\square, \{ 0, \infty \})$ to $(\mathbb{A}^1, \{ 0, 1 \})$, we have a closed formal subscheme of $\Spf (k[[t]] \{ y_1, y_2 \})$ defined by
$$
\left\{ \frac{y_1}{ y_1 -1} = 1+t, \ \  \frac{ y_2}{ y_2 -1} = t \right\},
$$
which gives
\begin{equation}\label{eqn:PU defect}
\{ 1 - (y_1 -1) t = 0, \ \ y_2 = (y_2 -1) t\}.
\end{equation}
However, the left hand side $1- (y_1 -1)t$ of \eqref{eqn:PU defect} is a unit in $k[[t]] \{ y_1, y_2 \}$ because the power series
$$
1 + (y_1 -1) t + (y_1 -1)^2 t^2 + \cdots + (y_1 -1)^i t^i + \cdots
$$
is convergent in $k[[t]] \{ y_1, y_2 \}$ and it is the inverse of $1- (y_1 -1)t$. In particular, the first equation of \eqref{eqn:PU defect} cannot hold. Hence $\mathfrak{Z}_2$ in $\Spf (k[[t]]) \times \square^2_k$ is empty, and such $\mathfrak{Z}_2$ is automatically eliminated in the formal scheme model, without forcibly imposing an extra condition as in the scheme model of \cite{PU Milnor}. 

Furthermore, later we will see that each integral admissible cycle in $\Spf (k[[t]]) \times \square_k ^n$ in the Milnor range is always finite over $k[[t]]$ in Corollary \ref{cor:finite quasi-finite}.
\qed
\end{exm}

Coming back to our definition, we want to form a cycle complex:

\begin{defn}\label{defn:HCG1-1}
Suppose $n \geq 1$. For each $1 \leq i \leq n$ and each $\epsilon \in \{ 0, \infty \}$, let $\iota_i ^{\epsilon}: \square_{\mathfrak{X}} ^{n-1} \hookrightarrow \square_{\mathfrak{X}} ^n$ be the closed immersion given by the single equation $\{ y_i = \epsilon \}$. For each integral cycle $\mathfrak{Z} \in \un{z} ^q (\mathfrak{X}, n)$, define the face $\partial_i ^{\epsilon} (\mathfrak{Z})$ to be the cycle associated to the intersection $\mathfrak{Z} \cap \{ y_i = \epsilon \}= [(\iota_i ^{\epsilon})^* \mathcal{O}_{\mathfrak{Z}}]$. One defines $\partial:= \sum_{i=1} ^n (-1)^i( \partial_i ^{\infty} - \partial_i ^0)$ and by the usual cubical formalism, we check that $\partial^2 = 0$.

This gives a complex $(\un{z}^q (\mathfrak{X}, \bullet), \partial)$ of abelian groups. The complex $(z^q (\mathfrak{X}, \bullet), \partial)$ is defined to be the quotient of $(\un{z}^q (\mathfrak{X}, \bullet), \partial)$ by the subcomplex $\un{z} ^q (\mathfrak{X}, \bullet)_{{\rm degn}}$ of degenerate cycles generated by pull-backs by the projections $\square_{\mathfrak{X}} ^n \to \square_{\mathfrak{X}} ^{n-1}$ that drop one of the coordinates. We define $\CH^q (\mathfrak{X}, n)$ to be the $n$-th homology of the complex $(z^q (\mathfrak{X}, \bullet), \partial)$.
\qed
\end{defn}

The group $\CH^q (\mathfrak{X}, n)$ is \emph{not} the primary object of studies of the article. In this article where we mostly consider $\mathfrak{X} = \widehat{X} = \Spf (k[[t]])$, we still need to take a quotient of $z^q (\widehat{X}, \bullet)$ by a subcomplex that gives the ``mod $Y$-equivalence". It will be given in Definition \ref{defn:mod Y equiv}.

\medskip

We mention that in the case of $\widehat{X} = \Spf (k[[t]])$ the following Tor independence property holds (see also Corollary \ref{cor:derived-underived}):

\begin{lem}\label{lem:Tor vanishing}
Let $Y= \Spec (k_m)$ and $\widehat{X}= \Spf (k[[t]])$. 
Let $\mathcal{F}$ be a coherent sheaf on $\square_{\widehat{X}} ^n$, whose associated cycle $[\mathcal{F}]$ belongs to ${z}^* (\widehat{X}, n)$. 

Then we have
${\Tor}_i ^{\mathcal{O}_{\square_{\widehat{X}}^n}} (\mathcal{F}, \mathcal{O}_{\square_{Y} ^n}) = 0$ for $i >0$.
\end{lem}

\begin{proof}As before, for $y_i' := y_i/ (y_i -1)$, note that we have the short exact sequence
$$
0 \to k[[t]] \{ y_1', \cdots, y_n' \} \overset{ \times t^m}{\to} k[[t]] \{ y_1', \cdots, y_n' \} \to k_m [ y_1', \cdots, y_n'] \to 0,
$$
which gives a free $\mathcal{O}_{\square_{\widehat{X}} ^n}$-resolution of $\mathcal{O}_{\square_Y ^n}$:
$$
0 \to \mathcal{O}_{\square_{\widehat{X}} ^n} \overset{ \times t^m}{\to} \mathcal{O}_{\square_{\widehat{X}} ^n} \to \mathcal{O}_{\square_Y^n} \to 0.
$$

Hence 
$$
\mathcal{F} \otimes_{\mathcal{O}_{\square_{\widehat{X}} ^n}} ^{\mathbf{L}} \mathcal{O}_{\square_Y^n} = \mathcal{F} \otimes _{\mathcal{O}_{\square_{\widehat{X}} ^n}} \left( \mathcal{O}_{\square_{\widehat{X}} ^n} \overset{ \times t^m}{\to} \mathcal{O}_{\square_{\widehat{X}} ^n} \right)= \left( \mathcal{F} \overset{\times t^m}{\to} \mathcal{F}\right),
$$
where the objects in the above complexes are placed in the homological degrees $+1$ and $0$, respectively. This shows that $\Tor_i ^{\mathcal{O}_{\square_{\widehat{X}}^n}} (\mathcal{F}, \mathcal{O}_{\square_{Y} ^n}) = 0$ for $i >1$, while $\Tor_1 ^{\mathcal{O}_{\square_{\widehat{X}}^n}} (\mathcal{F}, \mathcal{O}_{\square_{Y} ^n}) = \ker (\times t^m: \mathcal{F} \to \mathcal{F})$.

It remains to prove the vanishing of $\Tor_1$. There is a finite decreasing filtration of $\mathcal{O}_{\square_{\widehat{X}}^n}$-submodules
$$
0 = \mathcal{F}_r \subsetneq \mathcal{F}_{r-1} \subsetneq \cdots \subsetneq \mathcal{F}_0 = \mathcal{F}
$$
such that for each $0 \leq j \leq r-1$, we have $\mathcal{F}_j / \mathcal{F}_{j+1} \simeq \mathcal{O}_{\mathfrak{Z}_j}$ as $\mathcal{O}_{\square_{\widehat{X}} ^n}$-modules, for some integral closed formal subscheme $\mathfrak{Z}_j \subset \square_{\widehat{X}} ^n$ (see \cite[Lemma 01YF]{stacks}). 

\medskip

We prove the vanishing of $\Tor_1$ by induction on $r$. 

First consider the case when $r=1$, namely $\mathcal{F}= \mathcal{O}_{\mathfrak{Z}}$ for an integral cycle $\mathfrak{Z} \in z^* (\widehat{X}, n)$. Since $\mathfrak{Z}$ is integral and it intersects the closed subscheme $\square_Y ^n$ properly by the condition (\textbf{SF}) of Definition \ref{defn:HCG}, the element $t$ is not a zero-divisor on $\mathcal{O}_{\mathfrak{Z}}$. In particular, we have $\ker (\times t^m: \mathcal{O}_{\mathfrak{Z}} \to \mathcal{O}_{\mathfrak{Z}}) = 0$, thus,
\begin{equation}\label{eqn:single case Tor}
\Tor_1 ^{\mathcal{O}_{\square_{\widehat{X}} ^n}} (\mathcal{O}_{\mathfrak{Z}}, \mathcal{O}_{\square_Y^n}) = 0.
\end{equation}

When $r>1$, suppose we have the vanishing $\Tor_1 ^{\mathcal{O}_{\square_{\widehat{X}} ^n}} (\mathcal{F}_j, \mathcal{O}_{\square_Y^n}) =0 $ for all $j= 1, \cdots, r$. We want to prove its vanishing for $j=0$. Indeed, we have the short exact sequence of $\mathcal{O}_{\square_{\widehat{X}}^n}$-modules
$$
0 \to \mathcal{F}_1 \to \mathcal{F}_0 \to \mathcal{O}_{\mathfrak{Z}_0} \to 0
$$
from which we deduce part of the Tor long exact sequence
$$ \cdots \to \Tor_1 ^{\mathcal{O}_{\square_{\widehat{X}} ^n}} ( \mathcal{F}_1, \mathcal{O}_{\square_Y^n}) \to \Tor_1 ^{\mathcal{O}_{\square_{\widehat{X}} ^n}} ( \mathcal{F}_0, \mathcal{O}_{\square_Y^n}) \to \Tor_1 ^{\mathcal{O}_{\square_{\widehat{X}} ^n}} (\mathcal{O}_{\mathfrak{Z}_0}, \mathcal{O}_{\square_Y^n})\to \cdots.
$$
Here the first term vanishes by the induction hypothesis, and the third term vanishes by \eqref{eqn:single case Tor}. Hence the middle term also vanishes. Since $\mathcal{F}_0 = \mathcal{F}$, this proves the desired assertion of the lemma.
\end{proof}

\begin{remk}
As was also remarked by the referee, a bit more general statement than the one given in Lemma \ref{lem:Tor vanishing} may hold for various regular adic rings $A$. Since we do not need it in this article, we resist the temptation to generalize it.
\qed
\end{remk}

\begin{cor}\label{cor:derived-underived}
Let $Y= \Spec (k_m)$ and $\widehat{X}= \Spf (k[[t]])$. 
Let $\mathcal{A}$ be a coherent $\mathcal{O}_{\square_{\widehat{X}} ^n}$-algebra such that $[\mathcal{A}] \in z^* (\widehat{X}, n)$. 

Then we have an isomorphism of the graded $\mathcal{O}_{\square_Y^n}$-algebras
$$
\Tor_* ^{ \mathcal{O}_{\square_{\widehat{X}} ^n}} ( \mathcal{A}, \mathcal{O}_{\square_Y^n}) \simeq \Tor_0 ^{ \mathcal{O}_{\square_{\widehat{X}} ^n}} ( \mathcal{A}, \mathcal{O}_{\square_Y^n}).
$$
\end{cor}

\begin{proof}
The graded $\mathcal{O}_{\square_{Y}^n}$-algebra structure on $\Tor_* ^{ \mathcal{O}_{\square_{\widehat{X}} ^n}} ( \mathcal{A}, \mathcal{O}_{\square_Y^n}) $ is defined in \cite[Section 068G]{stacks}. 
Since all higher $\Tor$'s vanish by Lemma \ref{lem:Tor vanishing}, we deduce the corollary.
\end{proof}

\subsection{Some finite push-forwards}\label{sec:pf}
The primary case we are interested in is when $\mathfrak{X}$ is $ \widehat{X}=\Spf (k[[t]])$, but in the middle of our arguments later (\S \ref{sec:strong graph}) we need a bit of push-forwards for cycles on some formal schemes that are finite over $\widehat{X}$. In \S \ref{sec:pf}, we minimize our discussions only to those needed in this article.

Recall (Fujiwara-Kato \cite[Definition 4.7.1, p.341]{FK}) a morphism $f: \mathfrak{X} \to \mathfrak{Y}$ of noetherian formal schemes is \emph{proper} if it is separated of finite type (\cite[Definition 4.6.2, p.336]{FK}) and universally closed (\cite[Definition 4.5.5, p.334]{FK}). When $\mathfrak{X}$ and $\mathfrak{Y}$ are both affine, a proper morphism $f$ between them is necessarily finite; recall (\cite[Proposition 4.2.1, Definition 4.2.2, pp.324--325]{FK}) that a morphism $f: \mathfrak{X} \to \mathfrak{Y}$ of noetherian formal schemes is \emph{finite}, if for any affine open subset $V= \Spf (A)$ of $\mathfrak{Y}$, $f^{-1} (V)$ is affine of the form $\Spf (B)$ for a ring $B$ that is a finitely generated $A$-module.

\begin{defn}\label{defn:push-forward 1}
For a finite morphism $f: \mathfrak{X} \to \mathfrak{Y}$ of noetherian affine formal $k$-schemes, define the finite push-forward
on the naive groups (\S \ref{sec:naive cycle})
$$
f_*: \un{z}_d (\mathfrak{X}) \to \un{z}_d  (\mathfrak{Y})
$$
by sending an integral closed formal subscheme $\mathfrak{Z} \subset \mathfrak{X}$ of dimension $d$ to the $d$-dimensional cycle $[f_* \mathcal{O}_{\mathfrak{Z}}]_d$ associated to the coherent sheaf $f_* \mathcal{O}_{\mathfrak{Z}} $. Since $f$ is proper, the sheaf push-forward $f_* \mathcal{O}_{\mathfrak{Z}}$ is a coherent $\mathcal{O}_{\mathfrak{Y}}$-algebra.\qed
\end{defn}

Write $\mathfrak{X}= \Spf (A)$ and $\mathfrak{Y} = \Spf (B)$. By definition $\un{z}_d (\mathfrak{X}) = z_d (\Spec (A))$ and $\un{z}_d (\mathfrak{Y}) = z_d (\Spec (B))$ and $f$ is finite. Hence $f_*$ is identical to the usual finite push-forward on the cycles on the scheme $\Spec (A)$ via the associated finite morphism $f: \Spec (A) \to \Spec (B)$.

We deduce that the cycle push-forward and the sheaf push-forward are compatible:

\begin{lem}\label{lem:pf coh cy}
Let $f: \mathfrak{X} \to \mathfrak{Y}$ be a finite morphism of noetherian affine formal $k$-schemes. Let $\mathcal{F}$ be a coherent $\mathcal{O}_{\mathfrak{X}}$-module. 
Then
we have the equality of cycles
\begin{equation}\label{eqn:pf coh cy0}
f_* [ \mathcal{F}]= [ f_* \mathcal{F}] \in \un{z}_* (\mathfrak{Y}), 
\end{equation}
where the left hand side is the push-forward of the cycle $[\mathcal{F}]$, while the right hand side is the cycle $[f_* \mathcal{F}]$ associated to the push-forward sheaf $f_* \mathcal{F}$.
\end{lem}

\begin{proof}
Let $\mathfrak{X} = \Spf (A)$ and $\mathfrak{Y} = \Spf (B)$ so that the natural homomorphism $B \to A$ is finite. By construction we have $\un{z}_ * (\mathfrak{X}) = z_* (\Spec (A))$ and $\un{z}_* (\mathfrak{Y}) = z_* (\Spec (B))$, and we are given $[\mathcal{F}] \in z_* (\Spec (A))$. 
The push-forward $f_*$ on cycles is the push-forward of cycles on the scheme $\Spec (A)$. Hence the statement follows from the classical case of schemes, which is in \cite[Lemma 0EP3]{stacks}, for instance.
\end{proof}

The finite push-forwards as in Definition \ref{defn:push-forward 1} induce push-forwards on higher Chow cycles on noetherian affine formal $k$-schemes under an additional mild assumption:

\begin{lem}\label{lem:pf}

Let $f: \mathfrak{X} \to \mathfrak{Y}$ be a finite surjective morphism of equidimensional noetherian affine formal $k$-schemes. 

Then we have the induced push-forward morphism 
$$
f_*: z^q (\mathfrak{X}, \bullet) \to z^q (\mathfrak{Y}, \bullet)
$$
of complexes of abelian groups.
\end{lem}

\begin{proof}

Since $f$ is finite surjective, so is the morphism $f \times_k 1: \mathfrak{X} \times_k \square_k ^n \to \mathfrak{Y} \times_k \square_k ^n$, which we also denote by $f$. Furthermore, $\dim \ \mathfrak{X} = \dim \ \mathfrak{Y}$ and $\dim \ \mathfrak{X}_{\red} = \dim \ \mathfrak{Y}_{\red}$.

For an integral cycle $\mathfrak{Z} \in z^q (\mathfrak{X}, n)$, the cycle push-forward $f_*$ is defined to be $f_* (\mathfrak{Z}):= [ f_* \mathcal{O}_\mathfrak{Z}]_d$ as in Definition \ref{defn:push-forward 1}, where $d=\dim \ \mathfrak{Z}$. We need to check that $f_* (\mathfrak{Z})$ satisfies the conditions of Definition \ref{defn:HCG} over $\mathfrak{Y}$. Let $Z= \mathfrak{Z} \cap (\mathfrak{X} _{\red} \times \square_k ^n)$, the reduction by the largest ideal of definition of $\mathfrak{X} \times \square_k^n$.

\medskip

For a face $F \subset \square_k ^n$, we have $f_* (\mathfrak{Z}) \cap (\mathfrak{Y} \times F) = f_* (\mathfrak{Z} \cap (\mathfrak{X} \times F))$. Thus $\dim \ (f_* (\mathfrak{Z}) \cap ( \mathfrak{Y} \times F)) \leq \dim \ (\mathfrak{Z} \cap (\mathfrak{X} \times F))$. 

The general position condition (\textbf{GP}) of Definition \ref{defn:HCG} for $\mathfrak{Z}$ says
$$
\dim \ (\mathfrak{Z} \cap (\mathfrak{X} \times F)) \leq  \dim \ \mathfrak{X} \times F-q.
$$
So, we have
\begin{equation}\label{eqn:pf GP}
\dim \ (f_* (\mathfrak{Z}) \cap (\mathfrak{Y} \times F)) \leq \dim \ \mathfrak{X} \times F- q = \dim \ \mathfrak{Y} \times F -q,
\end{equation}
which is the condition (\textbf{GP}) for $f_* (\mathfrak{Z})$.

\medskip

For the special fiber condition (\textbf{SF}) for $f_* (\mathfrak{Z})$, first note that
\begin{equation}\label{eqn:SFSF0}
\dim \ f_* (\mathfrak{Z}) \cap (\mathfrak{Y}_{\red} \times F) \leq \dim \ \mathfrak{Z} \cap ( \mathfrak{X}_{f^{-1} (\mathcal{I}_0)} \times F),
\end{equation}
where $\mathfrak{X}_{f^{-1} (\mathcal{I}_0)}$ is the scheme associated to the ideal $f^{-1} (\mathcal{I}_0) \mathcal{O}_{\mathfrak{X}}$. 
Since $\mathfrak{Z} \cap (\mathfrak{X}_{\red} \times F)$ and $\mathfrak{Z} \cap (\mathfrak{X}_{f^{-1} (\mathcal{I}_0)} \times F)$ are both noetherian schemes, their dimensions as schemes are equal to the dimensions of their respective underlying noetherian topological spaces. But, their underlying topological spaces are both equal to $|\mathfrak{Z} \cap (\mathfrak{X} \times F)|$ (Lemma \ref{lem:TS}) so that we have
\begin{equation}\label{eqn:SFSF1}
 \dim \ \mathfrak{Z} \cap (\mathfrak{X}_{f^{-1} (\mathcal{I}_0)} \times F) =\dim \ \mathfrak{Z} \cap (\mathfrak{X}_{\red} \times F).
\end{equation}
On the other hand, by the condition (\textbf{SF}) for $\mathfrak{Z}$, we have
\begin{equation}\label{eqn:SFSF2}
\dim \ \mathfrak{Z} \cap (\mathfrak{X}_{\red} \times F) \leq \dim \ \mathfrak{X}_{\red} \times F -q = \dim \ \mathfrak{Y}_{\red} \times F - q.
\end{equation}
Combining \eqref{eqn:SFSF0}, \eqref{eqn:SFSF1}, and \eqref{eqn:SFSF2}, we deduce that
$$
\dim \ f_* (\mathfrak{Z}) \cap (\mathfrak{Y}_{\red} \times F) \leq \dim \ \mathfrak{Y}_{\red} \times F - q,
$$
which is the condition (\textbf{SF}) for $f_* (\mathfrak{Z})$. 

\medskip

Thus we have checked that $f_* (\mathfrak{Z}) \in z^q (\mathfrak{Y}, n)$, so that $f_*$ maps $z^q (\mathfrak{X}, n)$ to $ z^{q} (\mathfrak{Y}, n)$.

\medskip

For its compatibility with the codimension $1$ faces, take the face $F= F_i ^{\epsilon}$ for some $1 \leq i \leq n$ and $\epsilon \in \{ 0, \infty\}$ in the proof of the condition (\textbf{GP}) in the above. By \eqref{eqn:pf GP}, we obtain the square
$$
\xymatrix{ 
z^q (\mathfrak{X}, n) \ar[r] ^{\partial_i ^{\epsilon}} \ar[d] ^{f_*} & z^q (\mathfrak{X}, n-1) \ar[d] ^{f_*} \\
z^q (\mathfrak{Y}, n) \ar[r] ^{\partial_i ^{\epsilon}} & z^q (\mathfrak{Y}, n-1).}
$$
That this diagram commutes is checked exactly as done in \cite[Proposition (1.3)]{Bloch HC} using \cite[Theorem 6.2(a), p.98]{Fulton}. We omit details. This implies that $f_*$ commutes with $\partial$. This proves the lemma.
\end{proof}

\subsection{The mod $Y$-equivalence on cycles}\label{sec:mod Y}
Returning back to $Y:= \Spec (k_m)$ and $\widehat{X} = \Spf (k[[t]])$, we define the mod $Y$-equivalence on $z^q (\widehat{X}, n)$ as follows.

\begin{defn}\label{defn:mod Y equiv}
Under the above notations:
\begin{enumerate}
\item Let $\mathcal{R}^q (\widehat{X}, n)$ be the set of all coherent $\mathcal{O}_{\square_{\widehat{X}} ^n}$-algebras $\mathcal{A}$ such that the associated cycle $[\mathcal{A}] \in z^q (\widehat{X}, n)$.
\item Let $\mathcal{L}^q (\widehat{X}, Y, n)$ be the set of all pairs $(\mathcal{A}_1, \mathcal{A}_2)$ of coherent $\mathcal{O}_{\square_{\widehat{X}}^n}$-algebras with $\mathcal{A}_j \in \mathcal{R}^q (\widehat{X}, n)$ such that there is an isomorphism
\begin{equation}\label{eqn:iso algebra}
\mathcal{A}_1 \otimes _{\mathcal{O}_{\square_{\widehat{X}}^n} } \mathcal{O}_{\square_Y^n} \simeq \mathcal{A}_2 \otimes _{\mathcal{O}_{\square_{\widehat{X}}^n} } \mathcal{O}_{\square_Y^n} 
\end{equation}
of $\mathcal{O}_{\square_Y^n}$-algebras.
\item Let $\mathcal{M}^q (\widehat{X}, Y, n) \subset z^q (\widehat{X}, n)$ be the subgroup generated by the cycles of the form $[\mathcal{A}_1]-[\mathcal{A}_2]$ over all pairs $(\mathcal{A}_1, \mathcal{A}_2) \in \mathcal{L}^q (\widehat{X}, Y, n)$. 
\end{enumerate}

We say that two cycles \emph{$\mathfrak{Z}_1$ and $\mathfrak{Z}_2 \in z^q (\widehat{X}, n)$ are mod $Y$-equivalent} and write $\mathfrak{Z}_1 \sim_Y \mathfrak{Z}_2$, if $\mathfrak{Z}_1 - \mathfrak{Z}_2 \in \mathcal{M}^q (\widehat{X}, Y, n)$.\qed
\end{defn}

We have two remarks regarding Definition \ref{defn:mod Y equiv}.

\begin{remk}\label{remk:derived-underived}
Corollary \ref{cor:derived-underived} offers a hint on how one may generalize the mod $Y$-equivalence of Definition \ref{defn:mod Y equiv} to more general $k$-schemes of finite type to be discussed in \cite{Park general}.

For a general formal neighborhood, in \cite{Park general} the author uses some ideas from derived algebraic geometry. The isomorphism \eqref{eqn:iso algebra} is replaced by an isomorphism of a suitable pair of ``derived rings". While it is of a similar form, it is finer than the derived Milnor patching for perfect complexes studied by S. Landsburg \cite{Landsburg Duke}, which generalizes the classical Milnor patching of J. Milnor \cite{Milnor K}. 

It seems it is worth studying whether one can avoid using this heavy machine of derived algebraic geometry, by improving Lemma \ref{lem:Tor vanishing} to the fullest possible extent in the future.
\qed
\end{remk}

\begin{remk}\label{remk:mod Y=naive Y}
Let $I:= (t^m) \subset k[[t]]$. 
Instead of the above relation $\sim_Y$ on the cycles, one may ask whether it might be better to use the relation $\sim_I$ defined as follows: first declare $\mathfrak{Z}_1 \sim_I  \mathfrak{Z}_2$ for a pair of \emph{integral} cycles $\mathfrak{Z}_1$ and $ \mathfrak{Z}_2 \in z^q (\widehat{X}, n) $ when their structure sheaves $(\mathcal{O}_{\mathfrak{Z}_1}, \mathcal{O}_{\mathfrak{Z}_2})$ satisfy \eqref{eqn:iso algebra} of Definition \ref{defn:mod Y equiv}, namely
$$
\mathcal{O}_{\mathfrak{Z}_1} \otimes _{\mathcal{O}_{\square_{\widehat{X}} ^n}} \mathcal{O}_{\square_Y^n} \simeq
\mathcal{O}_{\mathfrak{Z}_2} \otimes _{\mathcal{O}_{\square_{\widehat{X}} ^n}} \mathcal{O}_{\square_Y^n}.
$$

 Define the subgroup $\mathcal{N} ^q (\widehat{X}, Y, n) \subset z^q (\widehat{X},  n)$ generated by $[\mathfrak{Z}_1]- [\mathfrak{Z}_2]$ for such pairs $(\mathcal{O}_{\mathfrak{Z}_1}, \mathcal{O}_{\mathfrak{Z}_2})$. We define $\sim_I$ on $z^q (\widehat{X}, n)$ using this subgroup $\mathcal{N}^q (\widehat{X}, Y, n)$. It is similar to the version taken in \cite{PU Milnor}, called the ``mod $t^m$-equivalence" there.

A benefit of $\sim_I$ is that it is relatively easier to work with; when one wishes to check whether certain operations on cycles respect the equivalence, one can test them just on pairs of integral cycles. A draw-back of this is that, the property of a scheme being integral is not preserved well under base change in general (see, e.g. EGA ${\rm IV}_2$ \cite[D\'efinition (4.6.2), p.68]{EGA4-4} for the definition of geometric integrality, which does behave well under base change), so that this approach counting on putting relations on pairs integral cycles is unlikely to be generalizable in a functorial way, in general.

However, at least in the Milnor range (i.e. $q=n$), we will show that our definition of $\sim_Y$ via $\mathcal{M}^n (\widehat{X}, Y,n)$ in Definition \ref{defn:mod Y equiv} and $\sim_I$ in the above paragraph via $\mathcal{N}^n (\widehat{X}, Y, n)$ are equivalent. The proof requires a few properties of the generating cycles in $z^n (\widehat{X}, n)$, so that it is postponed until Lemma \ref{lem:mod Y=naive Y}. This result will be used in \S \ref{sec:strong graph} and \S \ref{sec:5}.\qed
\end{remk}

Coming back to our construction, we observe:

\begin{lem}\label{lem:mod Y boundary}
Let $Y= \Spec (k_m)$ and $\widehat{X}= \Spf (k[[t]])$. Let $n \geq 1$.
For $1 \leq i \leq n$ and $\epsilon \in \{0, \infty \}$, we have 
$$\partial_i ^{\epsilon} \mathcal{M}^q (\widehat{X}, Y, n) \subset \mathcal{M}^q (\widehat{X}, Y, n-1).
$$

In particular, $\mathcal{M}^q (\widehat{X}, Y, \bullet) \subset z^q (\widehat{X}, \bullet)$ is a subcomplex.
\end{lem}

\begin{proof}
Let $\iota_i ^{\epsilon}: \square_{\widehat{X}} ^{n-1} \hookrightarrow \square_{\widehat{X}} ^n$ be the closed immersion given by $\{ y_i = \epsilon \}$. 

When $\mathcal{A} \in \mathcal{R}^q (\widehat{X}, n)$, we have $[\mathcal{A}] \in z^q (\widehat{X}, n)$. Here $ [( \iota_i ^{\epsilon})^* (\mathcal{A})]=\partial_i ^{\epsilon} [ \mathcal{A}]  \in z^q (\widehat{X}, n-1)$, so that $(\iota_i ^{\epsilon})^* \mathcal{A} \in \mathcal{R}^q (\widehat{X}, n-1)$.

When $(\mathcal{A}_1, \mathcal{A}_2) \in \mathcal{L}^q (\widehat{X}, Y, n)$, we have $\mathcal{A}_j \in \mathcal{R}^q (\widehat{X}, n)$ and there is an isomorphism
\begin{equation}\label{eqn:iso algebra 0 bdry}
\mathcal{A}_1 \otimes _{\mathcal{O}_{\square_{\widehat{X}}^n} } \mathcal{O}_{\square_Y^n} \simeq \mathcal{A}_2 \otimes _{\mathcal{O}_{\square_{\widehat{X}}^n} } \mathcal{O}_{\square_Y^n} 
\end{equation}
of $\mathcal{O}_{\square_Y^n}$-algebras. We already saw that $(\iota_i ^{\epsilon})^* \mathcal{A}_j \in \mathcal{R}^q (\widehat{X}, n-1)$. Applying $(\iota_i ^{\epsilon})^*$ to \eqref{eqn:iso algebra 0 bdry}, we deduce an isomorphism
$$
(\iota_i ^{\epsilon})^*\mathcal{A}_1 \otimes _{\mathcal{O}_{\square_{\widehat{X}}^{n-1}} } \mathcal{O}_{\square_Y^{n-1}} \simeq (\iota_i ^{\epsilon})^*\mathcal{A}_2 \otimes _{\mathcal{O}_{\square_{\widehat{X}}^{n-1}} } \mathcal{O}_{\square_Y^{n-1}} 
$$
of $\mathcal{O}_{\square_Y^{n-1}}$-algebras. This means $((\iota_i ^{\epsilon})^* \mathcal{A}_1, (\iota_i ^{\epsilon})^*\mathcal{A}_2) \in \mathcal{L}^q (\widehat{X}, Y, n-1)$. Thus
$$
\partial_i ^{\epsilon} ([\mathcal{A}_1] - [\mathcal{A}_2]) = [ (\iota_i ^{\epsilon})^* \mathcal{A}_1] -  [ (\iota_i ^{\epsilon})^* \mathcal{A}_2] \in \mathcal{M}^q (\widehat{X}, Y, n-1).
$$

Since the cycles of the form $[\mathcal{A}_1] - [\mathcal{A}_2]$ generate $\mathcal{M}^q (\widehat{X}, Y, n)$, this implies that $\partial_i ^{\epsilon} \mathcal{M}^q (\widehat{X}, Y, n) \subset \mathcal{M}^q (\widehat{X}, Y, n-1)$ as desired.
\end{proof}

\begin{defn}\label{defn:defn Milnor}
Let $m \geq 2$.
Let $Y= \Spec (k_m)$ and $\widehat{X} = \Spf (k[[t]])$. Define
$$
z^q (\widehat{X} \mod Y, \bullet):= \frac{ z^q (\widehat{X}, \bullet)}{\mathcal{M}^q (\widehat{X}, Y, \bullet)},
$$
and define $\CH^q (\widehat{X} \mod Y, n):= {\rm H}_n (z^q (\widehat{X} \mod Y, \bullet)).$ We denote this group also by $\BGH^q (k_m,n)= \BGH^q (\Spec (k_m), n)$ in the boldface letters, to distinguish it from the higher Chow group of S. Bloch \cite{Bloch HC}.\qed
\end{defn}

\medskip

The following special range is easy to compute. It answers the part (1) of the first main theorem, Theorem \ref{thm:main intro}:

\begin{prop}\label{prop:q>n}
For $q > n$, $z^q (\widehat{X} \mod Y, n) = 0$. 
In particular, we have 
$$
\CH^q (\widehat{X} \mod Y, n) =\BGH^q (k_m, n) = 0.
$$
\end{prop}

\begin{proof}
Note that $\dim \ \widehat{X} \times \square^n = n+1$. Thus if $q \geq n+2$, due to the dimension reason, we have $z^q (\widehat{X}, n) = 0$. Thus $z^q (\widehat{X} \mod Y, n) = 0$ as well.

Suppose $q=n+1$. Let $\mathfrak{Z} \in z^{n+1} (\widehat{X}, n)$ be an integral cycle. Note that $\widehat{X}_{\red} = \Spec (k)$. By the special fiber condition (\textbf{SF}) of Definition \ref{defn:HCG}, the intersection $\mathfrak{Z} \cap (\Spec (k) \times \square^n)$ has the codimension $\geq n+1$ in the space $\Spec (k) \times \square^n$ of dimension $n$. This shows the intersection $\mathfrak{Z} \cap (\Spec (k) \times \square^n)$ is empty. On the other hand, we have 
$$
| \mathfrak{Z} \cap (\Spec (k) \times \square^n_k) | \subset | \mathfrak{Z} \cap (\Spec (k_m) \times \square_k ^n)| \subset | \mathfrak{Z} \cap (\widehat{X} \times \square_k ^n)|,
$$
and the first and the third topological spaces are equal (Lemma \ref{lem:TS}). Thus $\mathfrak{Z} \cap (\Spec (k_m) \times \square^n) = \emptyset$ as well. It means
$$
\mathcal{O}_{\mathfrak{Z}} \otimes_{\mathcal{O}_{\square_{\widehat{X}}^n}} \mathcal{O}_{\square_Y^n} = 0  =  \mathcal{O}_{\emptyset} \otimes_{\mathcal{O}_{\square_{\widehat{X}}^n}} \mathcal{O}_{\square_Y^n},
$$
thus the cycle $\mathfrak{Z}$ is mod $Y$-equivalent to the empty scheme $\emptyset$. Thus $[\mathfrak{Z}] \in \mathcal{M}^{n+1} (\widehat{X}, Y, n)$ and the class of $\mathfrak{Z}$ in $z^{n+1} (\widehat{X} \mod Y, n)$ is $0$. This shows $z^{n+1} (\widehat{X} \mod Y, n) = 0$.
\end{proof}

\begin{remk}
If we specialize to the Milnor range $q=n$ in Definition \ref{defn:defn Milnor}, the group $\CH^n (\widehat{X} \mod Y, n)$ can be expressed as
\begin{equation}\label{eqn:partial0}
\CH^n (\widehat{X} \mod Y, n) = \frac{ z^n (\widehat{X} \mod Y, n)}{ \partial ( z^n (\widehat{X} \mod Y, n+1))}
\end{equation}
because $z^{n} (\widehat{X} \mod Y, n-1) = 0$ by Proposition \ref{prop:q>n} so that
$$ \ker (\partial: z^n (\widehat{X} \mod Y, n) \to z^n (\widehat{X} \mod Y, n-1)) = z^n (\widehat{X} \mod Y, n).
$$
 Thus we may rewrite \eqref{eqn:partial0} also as
\begin{equation}\label{eqn:defn Milnor}
\CH^n (\widehat{X} \mod Y, n):= \frac{ z^n (\widehat{X}, n) }{ \partial ( z^n (\widehat{X}, n+1) ) + \mathcal{M}^n (\widehat{X}, Y, n)}.
\end{equation}

One may wonder whether the numerator in \eqref{eqn:defn Milnor} is also equal to $\ker (\partial : z^n (\widehat{X}, n) \to z^n (\widehat{X}, n-1))$. The answer is yes, and we will see it in Lemma \ref{lem:no face}. \qed
\end{remk}

 \subsection{Tate algebras}\label{sec:Tate}
 
 In \S \ref{sec:Tate}, we recall some basic definitions and results related to the Tate algebras of J. Tate \cite{Tate}. We will use them to deal with the ``generic fibers" of the cycles in this article. For general references on the Tate algebras, there are a few books in the literature, e.g. Bosch-G\"untzer-Remmert \cite{BGR}, Bosch \cite{Bosch},  Fresnel-van der Put \cite{FvdP}, Fujiwara-Kato \cite{FK}. At some places, Tate algebras are also called the standard affinoid algebras or the free affinoid algebras, as well.

For a complete discrete valued field $K$ with a nontrivial non-archimedean norm, such as $k\llp t\rlp$, the Tate algebra $T_n$ over $K$ with $n$-variables $z_1, \cdots, z_n$ is defined to be the subring $K \{ z_1, \cdots, z_n \} \subset K[[z_1, \cdots, z_n ]]$ consisting of the restricted formal power series, i.e.
$$ 
p=\sum_{I= (i_1, \cdots, i_n ), i_j \geq 0}  a_I z^I,
$$
with $z^I: = z_1 ^{i_1} \cdots z_n ^{i_n}$ such that $|a_I| \to 0$ as $|I| = i_1 + \cdots + i_n \to \infty$. The norm on $K$ extends to the Gau{\ss} norm on $T_n$ by taking $|p|:= \max | a_I |$, making $T_n$ a Banach algebra over $K$ of at most countable type. See e.g. any one of \cite[\S 5.1.1, p.192]{BGR},  \cite[\S 2.2, Definition 2.2, p.13]{Bosch}, \cite[\S 3.1, p.46]{FvdP}, or \cite[\S 9.3(a), p.227]{FK}.

The following is a compilation of a few facts on Tate algebras, that we need:

\begin{lem}\label{lem:Tate}
Let $K$ be a complete discrete valued field with a nontrivial non-archimedean norm, and let $T_n$ be the Tate algebra over $K$ in $n$-variables. Then:

\begin{enumerate}
\item $T_n$ is a noetherian regular UFD of Krull dimension $n$.
\item $T_1$ is a Euclidean domain, in particular a PID.
\item For each maximal ideal $\mathfrak{m} \subset T_n$, the injective ring homomorphism $K \hookrightarrow T_n/\mathfrak{m}$ gives a finite extension of fields.
\item For any finite extension $K \subset L$ of fields, there exists a unique extension of the discrete valuation on $K$ to $L$.
\item Every $f \in T_1$ can be written uniquely as $f = u g$, where $u \in T_1 ^{\times}$ and $g \in K^{\circ} [z_1]$ for the valuation ring $K^{\circ} = \{ x \in K \ | \ |x| \leq 1 \}$.
\end{enumerate}
\end{lem}

\begin{proof}
(1) is scattered at a few places. Collect them from \cite[\S 5.2.6, Theorem 1, p.207]{BGR}, \cite[\S 2.2, Proposition 14, 15, p.20, Proposition 17, p.22]{Bosch}, \cite[Theorem 3.2.1-(1),(2), p.48]{FvdP}, and \cite[Exercise 0.9.4, p.23, Proposition 9.3.9, p.229]{FK}.

(2) : see \cite[\S 2.2, Corollary 10, p.19]{Bosch}. 

(3) : see \cite[Corollary 12, p.19]{Bosch} or \cite[Theorem 3.2.1-(5), p.48]{FvdP}.

(4) : see e.g. H. Hasse \cite[Ch.12, Prolongation theorem, p.183]{Hasse}. 

(5) : see \cite[Exercises 3.2.2-(1), p.50]{FvdP}, as a consequence of the Weierstra{\ss} division theorem \cite[Theorem 3.1.1, p.46]{FvdP}, or e.g. \cite[\S 5.2.1, Theorem 2, p. 200, \S 5.2.2, Theorem 1, p.201]{BGR},  \cite[\S 2.2, Theorem 8, p.17, Corollary 9, p.18]{Bosch}, \cite[Exercise 0.A.3, p.236]{FK}. 
\end{proof}

\section{On the generator cycles}\label{sec:generators}

We continue to use the notations $Y= \Spec (k_m)$ and $\widehat{X}= \Spf (k[[t]])$. 

In \S \ref{sec:generators}, we study a few general aspects of the cycles in the group $z^n (\widehat{X}, n)$. These cycles generate the groups $\CH^n (\widehat{X}, n)$ and $\BGH^n (\Spec (k_m), n)$.

The arguments here are a hybrid of basic commutative algebra, basic rigid analytic geometry, and algebraic geometry of cycles on formal schemes. An essential observation is that the ``generic fiber" of an integral cycle in $z^n (\widehat{X}, n)$ induces a maximal ideal of a Tate algebra over the Laurent field $k\llp t \rlp$. Some results of rudimentary rigid analytic geometry then allow us to obtain  ``triangular polynomial generators" of the ideal that define the given integral cycle over $\Spf (k[[t]])$ in \S \ref{sec:appendix 02}.

\subsection{Algebraicity and generic fibers}\label{sec:generic fiber}
For a complete discrete valuation field $K$ with a nontrivial non-archimedean norm, consider $T_n= K\{ z_1, \cdots, z_n \}$, the Tate algebra in $n$ variables $z_1, \cdots, z_n$. Let's start with the standard fact that each maximal ideal $\mathfrak{m} \subset T_n$ of the Tate algebra is generated by $n$ elements (in particular it is a complete intersection); see S. Bosch \cite[\S 2.2, Proposition 17, p.22]{Bosch}. Here, these $n$ generators can be chosen from the subring $K[z_1,  \cdots, z_n ] \subset T_n$ of polynomials in $z_1, \cdots, z_n$; see Fresnel-van der Put \cite[Exercises 3.2.2-(2), p.50]{FvdP} or Bosch-G\"untzer-Remmert \cite[\S 7.1.1, Proposition 3, p.261]{BGR}. 

We improve this fact a bit further as follows. It might be of independent interest:

\begin{lem}\label{lem:Bosch 17}
Let $K$ be a complete discrete valuation field with a nontrivial non-archimedean norm, and $K^{\circ}$ be its valuation ring. Let $T_n= K \{ z_1, \cdots, z_n \}$ be the Tate algebra over $K$, and let $\mathfrak{m} \subset T_n$ be a maximal ideal. 

Then we have a sequence of \textbf{polynomials} in $z_1, \cdots, z_n$ 
$$
\tuborg p_1 (z_1) \in K^{\circ} [ z_1], \\
p_2 (z_1, z_2) \in K [ z_1, z_2], \\
\vdots \\
p_n (z_1, \cdots, z_n) \in K[z_1, \cdots, z_n],
\sluttuborg
$$
where
\begin{enumerate}
\item each $p_i$ is monic in $z_i$ for $1 \leq i \leq n$,
\item for $2 \leq i \leq n$, when $p_i$ is regarded as a polynomial in $z_i$, its coefficients in $ K [ z_1, \cdots, z_{i-1}]$ have their $z_{j}$-degrees strictly less than $\deg_{z_j} (p_j)$ for all $1 \leq j < i$, and
\item $(p_1, \cdots, p_n) = (p_1, \cdots , p_n) T_n = \mathfrak{m}$.
\end{enumerate}
\end{lem}

\begin{proof}
The proof is in part based on the argument of S. Bosch \cite[\S 2.2, Proposition 17, p.22]{Bosch}, with some improvements to deduce our stronger statements.

For the sequence of inclusions $T_1 \subset T_2 \subset \cdots \subset T_n$, consider the prime ideals $\mathfrak{m}_i:= T_i \cap \mathfrak{m}$ for $1 \leq i \leq n$. Here $\mathfrak{m}_n = \mathfrak{m}$. These induce injections
$$
K \hookrightarrow T_1 / \mathfrak{m}_1 \hookrightarrow \cdots \hookrightarrow T_n / \mathfrak{m}_n,
$$
where $T_n/ \mathfrak{m}_n$ is a finite extension of $K$ of fields (Lemma \ref{lem:Tate}-(3)). Hence each intermediate $T_i/ \mathfrak{m}_i$ is also a field, thus $\mathfrak{m}_i \subset T_i$ is a maximal ideal for $1 \leq i \leq n$.

\medskip

When $n=1$, the ideal $\mathfrak{m}_1 \subset T_1$ is maximal, while $T_1$ is a PID by Lemma \ref{lem:Tate}-(2). Hence we have $(f) = \mathfrak{m}_1$ for some nonzero $f \in T_1$. By Lemma \ref{lem:Tate}-(5), we have $f = u g$ for some $u \in T_1 ^{\times}$ and a monic polynomial $g \in K^{\circ} [z_1]$. Thus $(f) = (g) = \mathfrak{m}_1$. Taking $p_1 := g$ answers the lemma in this case.

\medskip

Now suppose $n \geq 2$, and suppose the lemma holds for all positive integers $< n$, so that we have 
$$
p_1 (z_1) \in K^{\circ} [z_1], p_2 (z_1, z_2) \in K[z_1, z_2], \cdots, p_{n-1} (z_1, \cdots, z_{n-1})  \in K[z_1, \cdots, z_{n-1}]
$$ 
satisfying the conclusions of the lemma for the maximal ideal $\mathfrak{m}_{n-1} \subset T_{n-1}$.

\medskip

Note that (e.g. from \emph{loc.cit}) we have the following commutative diagram of continuous $K$-algebra homomorphisms
\begin{equation}\label{eqn:Bosch 17-1}
\xymatrix{
T_{n-1} \otimes_{T_{n-1}} T_n = T_{n-1} \{ z_n \} \ar@{=}[rr] \ar[d] _{{\rm can} \otimes {\rm Id}_{T_n} =:\phi' } & & T_n \ar[d] ^{\phi} \\
(T_{n-1}/ \mathfrak{m}_{n-1}) \otimes_{T_{n-1}} T_n = (T_{n-1}/ \mathfrak{m}_{n-1}) \{ z_n \} \ar[rr] ^{\ \ \ \ \ \ \ \ \ \ \ \ \ \ \ \ \ \ \   \pi} & & T_n / \mathfrak{m},}
\end{equation}
where the vertical maps $\phi, \phi'$ are given by the reduction mod $\mathfrak{m}$ and $\mathfrak{m}_{n-1}$, respectively, and $\pi$ maps $z_n$ to its residue class in $T_n/\mathfrak{m}$. Note that both of the maps $\phi'$ and $\pi$ are surjective.

Take $\ker (\pi)$. Since $T_n/ \mathfrak{m}$ is a field and $\pi$ is surjective, by the first isomorphism theorem $\ker (\pi) \subset (T _{n-1}/ \mathfrak{m}_{n-1})\{ z_n \}$ is a maximal ideal. On the other hand, by Lemma \ref{lem:Tate}-(4), the field $T _{n-1}/ \mathfrak{m}_{n-1}$ is also a complete discrete valued field with respect to the unique norm extending the norm on $K$. Hence by the case $n=1$ we proved already, there exists a polynomial $\bar{p}_{n} \in ( T _{n-1}/ \mathfrak{m}_{n-1}) [ z_n]$ monic in $z_n$, such that $(\bar{p}_n) = \ker (\pi)$.

Since $\phi'$ is surjective, we can choose a lifting $p_n (z_1, \cdots, z_n) \in T_{n-1} [ z_n] \subset T_{n-1} \{ z_n \}$ of $\bar{p}_n$ which is still a monic polynomial in $z_n$. 

By the commutativity of the diagram \eqref{eqn:Bosch 17-1}, we have $(p_1, \cdots, p_n) = \mathfrak{m}$, but one issue yet to resolve is that, while $p_n (z_1, \cdots, z_n)$ is in $T_{n-1} [z_n]$, it may not be yet in $K[z_1, \cdots, z_{n-1}][z_n]$. 

To improve this, when $N:= \deg_{z_n} (p_n)$, write
$$
p_n = \sum_{i=0} ^N \alpha_i z_n ^i \in T_{n-1} [ z_n],
$$
for $\alpha_i \in T_{n-1}$, $0 \leq i \leq N$ with $\alpha_N = 1$. By the induction hypothesis $p_{n-1}$ is a monic polynomial in $z_{n-1}$. By the Weierstra{\ss} division theorem (Fresnel-van der Put \cite[Theorem 3.1.1, p.46]{FvdP}), for $0 \leq i < N$, we can find $q_{i} \in T_{n-1}$ and $r_i \in T_{n-2} [ z_{n-1}]$ such that
$$
\alpha_i = q_i p_{n-1} + r_i,
$$
where $r_i$ is a polynomial in $z_{n-1}$ with $\deg_{z_{n-1}} ( r_i) < \deg_{z_{n-1}} ( p_{n-1}).$ We let $r_N = 1$. Then for
$$
p_n ' := \sum_{i=1} ^N  r_i z_n ^i \in T_{n-2} [ z_{n-1}][z_n]
$$
we have $\mathfrak{m}=(p_1, \cdots, p_{n-1}, p_n) = (p_1, \cdots, p_{n-1}, p_n')$. Hence we may replace $p_n$ by $p_n'$. We can inductively apply the Weierstra{\ss} division theorem with the divisions by $p_{n-2}, p_{n-3}, \cdots, p_1$, and after these backward inductive replacements, in finite steps we obtain $p_n \in K [z_1, \cdots, z_n]$. This proves the lemma.
\end{proof}

As an immediate application, we deduce that our cycles in the Milnor range can be seen as separated $k[[t]]$-schemes of finite presentation. It will be further improved in the next subsections.

\begin{cor}\label{cor:pretriangular}
Let $\mathfrak{Z} \in z^n (\Spf (k[[t]]), n)$ be an integral cycle. Let $y_i':= y_i / (y_i -1)$ using the automorphism $y_i \mapsto y_i'$ of $\mathbb{P}^1$. 

Then there are \textbf{polynomials} in $k[[t]] [ y_1', \cdots, y_n' ]$ of the form
\begin{equation}\label{eqn:pretriangular}
\tuborg
P_1 (y_1') \in k[[t]] [y_1'],\\
\vdots \\
P_n (y_1', \cdots, y_n' ) \in k[[t]] [y_1', \cdots, y_n'],\sluttuborg
\end{equation}
such that the ideal $(P_1 (y_1'), \cdots, P_n (y_1', \cdots, y_n'))$ defines $\mathfrak{Z}$.

In particular, $\mathfrak{Z}$ can be also regarded as a separated $k[[t]]$-scheme of finite presentation.
\end{cor}

\begin{proof}
Let $K:= k\llp t\rlp  = {\rm Frac} (k[[t]])$. Write $z_i:= y_i '$ for notational simplicity as well as to be consistent with the notation of Lemma \ref{lem:Bosch 17}.

The cycle $\mathfrak{Z}$ is given by a height $n$ prime ideal $I (\mathfrak{Z})$ of the Tate algebra $T_n=k[[t]]\{ z_1, \cdots, z_n \}$. So, we have the canonical homomorphisms
$$
k[[t]] \hookrightarrow k[[t]] \{ z_1, \cdots, z_n \} \twoheadrightarrow  k[[t]] \{ z_1, \cdots, z_n \}/ I(\mathfrak{Z}),
$$
and, we may regard $\mathfrak{Z}$ as a scheme over $\Spec (k[[t]])$ as well.

Let $\eta \in \Spec (k[[t]])$ be the generic point. The generic fiber $\mathfrak{Z}_{\eta}$ is obtained by the flat base change via $\eta \to \Spec (k[[t]])$, and it gives a height $n$ prime ideal of the Tate algebra $T_n $. Since the Krull dimension of $T_n$ is $n$ (Lemma \ref{lem:Tate}-(1)), this height $n$ prime ideal is maximal. Hence the generic fiber $\mathfrak{Z} _{\eta}$ is defined by a maximal ideal $\mathfrak{m} \subset T_n$. 

By Lemma \ref{lem:Bosch 17}, we then have a set of generators $p_1 (z_1) \in K^{\circ}[z_1]$, $p_2 (z_1, z_2) \in K[z_1, z_2]$, $\cdots$, $ p_n (z_1, \cdots, z_n) \in K [ z_1, \cdots, z_n]$ with the properties there.

 To obtain the generators for its closure $\mathfrak{Z}$, we clear the denominators of the coefficients of the terms of $p_i$ by the l.c.m. in $k[[t]]$. Since $k[[t]]$ is a UFD, this is possible. So obtained elements are denoted by $P_1 (z_1) \in k[[t]][z_1] , \cdots, P_n  (z_1, \cdots, z_n) \in k[[t]] [ z_1, \cdots, z_n ]$.
\end{proof}

In \S \ref{sec:appendix 01}, after proving the finiteness, we will show that the above can be also stated for a sequence of polynomials in $k[[t]]\{ y_1, \cdots, y_n \}$ (not just for $k[[t]] \{ y_1 ', \cdots, y_n '\}$), where the polynomials $P_i$ can satisfy more requirements.

\subsection{Finiteness}\label{sec:appendix 01}

In \S \ref{sec:appendix 01}, we study the generators of $z^n (\Spf (k[[t]]), n)$ from the algebraic perspective to obtain additional properties.

\begin{lem}\label{lem:Milnor quasi-finite}
Let $\mathfrak{Z} \in z^n (\Spf (k[[t]]), n)$ be an integral cycle. 

Then $\mathfrak{Z}$ regarded as a scheme over $k[[t]]$, is quasi-finite over $k[[t]]$. 
\end{lem}

\begin{proof}
By Corollary \ref{cor:pretriangular}, we may regard $\mathfrak{Z}$ as a scheme over $k[[t]]$. 

We check that for each $\mathfrak{p} \in \Spec (k[[t]])$, the fiber $\mathfrak{Z}_{\mathfrak{p}}$ is a finite set. 

As before, via a suitable automorphism of $\mathbb{P}^1$, use $y_i' = y_i / (y_i -1)$ for $1 \leq i \leq n$. The generic fiber $\mathfrak{Z}_{\eta}$ is given by a maximal ideal of the Tate algebra $T_n = k\llp t\rlp \{ y_1', \cdots, y_n'\}$. Hence $|\mathfrak{Z}_{\eta}|= 1 < \infty$.

For the closed point $\mathfrak{m}\in \Spec (k[[t]])$, by the special fiber condition (\textbf{SF}) of Definition \ref{defn:HCG}, this cycle intersects $\mathfrak{m} \times \square^n_k$ properly. Hence $\mathfrak{Z}_{\mathfrak{m}}$ gives a closed subscheme of codimension $\geq n$ in $\square_{\kappa (\mathfrak{m})}^n = \square_k ^n$. Since $\dim \ \square_k ^n = n$, we have $\dim \ (\mathfrak{Z}_{\mathfrak{m}}) \leq 0$ so that $| \mathfrak{Z}_{\mathfrak{m}} |< \infty$ again. 

Thus we proved that $\mathfrak{Z} \to \Spec (k[[t]])$ is quasi-finite, as desired.
\end{proof}

We will refine Lemma \ref{lem:Milnor quasi-finite} even further. To do so, recall the following (See Conrad-Stein \cite[Lemma 8.3]{CS}. See also EGA ${\rm IV}_4$ \cite[Th\'eor\`eme (18.5.11)-(c), p.130]{EGA4-4}):

\begin{lem}\label{lem:fqf over hensel}
Let $A$ be a henselian local ring, and let $G$ be a quasi-finite separated $A$-scheme of finite presentation. 

Then there is a unique decomposition $G= G_f \coprod G'$ into disjoint clopen pieces such that $G_f \to \Spec (A)$ is finite, and $G'$ has the empty special fiber.
\end{lem}

An interesting and important point is that the nonzero integral cycles in $z^n (\widehat{X}, n)$ are all finite in the Milnor range:

\begin{cor}\label{cor:finite quasi-finite}
Let $\mathfrak{Z} \in z^n (\Spf (k[[t]]), n)$ be a nonzero integral cycle. Then $\mathfrak{Z}$ is finite over $k[[t]]$.
\end{cor}

\begin{proof}
The given integral cycle $\mathfrak{Z}$ is quasi-finite over the henselian local ring $k[[t]]$ by Lemma \ref{lem:Milnor quasi-finite}. Furthermore, it can be regarded as a separated $k[[t]]$-scheme of finite presentation by Corollary \ref{cor:pretriangular}.

Hence we may apply Lemma \ref{lem:fqf over hensel}. By this, $\mathfrak{Z}$ can be written as the disjoint union of the finite part $\mathfrak{Z}_f$, and the part $\mathfrak{Z}'$ that has the empty special fiber. But, $\mathfrak{Z}$ is integral, so we have either (a) $\mathfrak{Z} = \mathfrak{Z}_f$, in which case $\mathfrak{Z} \to \Spec (k[[t]])$ is finite, or (b) $\mathfrak{Z}=\mathfrak{Z}'$, in which case $\mathfrak{Z} \to \Spec (k[[t]])$ is not finite, and the special fiber is empty.

Since the cycle $\mathfrak{Z}$ \emph{does have} the nonempty special fiber (Lemma \ref{lem:TS}), the above case (b) cannot happen. Hence the case (a) is the only possibility and $\mathfrak{Z} \to \Spec (k[[t]])$ is finite. 
\end{proof}

Being finite over $k[[t]]$ have a few important consequences. The first is on intersection with the faces:

\begin{lem}\label{lem:no face}
Let $\mathfrak{Z} \in z^n (\Spf (k[[t]]), n)$ be an integral cycle. Then $\mathfrak{Z}$ does not intersect any proper face of $\Spf (k[[t]]) \times \square^n$. In particular, $\partial (\mathfrak{Z}) = 0$. 
\end{lem}

\begin{proof}
We imitate the argument of Krishna-Park \cite[Lemma 2.21, p.1005]{KP sfs}. Let $\widehat{X}= \Spf (k[[t]])$. Let $F \subsetneq \square_k ^n$ be a proper face. We show that $\mathfrak{Z} \cap (\widehat{X} \times F) = \emptyset$.

Suppose not. As before, we regard $\mathfrak{Z}$ as a scheme over $\Spec (k[[t]])$. By Corollary \ref{cor:finite quasi-finite}, the morphism $\mathfrak{Z} \to \Spec (k[[t]])$ is finite. Consider the composite of finite morphisms
$$ 
\mathfrak{Z} \cap (\widehat{X} \times F) \hookrightarrow \mathfrak{Z} \to \Spec (k[[t]]).
$$
Its image in $\Spec (k[[t]])$ is closed. Since we supposed that $\mathfrak{Z} \cap (\widehat{X} \times F)$ is nonempty, the image must intersect the unique closed point $\mathfrak{m}$ of $\Spec (k[[t]])$. Hence $\mathfrak{Z} \cap (\mathfrak{m} \times F) \not = \emptyset$.

On the other hand, by the special fiber condition (\textbf{SF}) of Definition \ref{defn:HCG} for $\mathfrak{Z}$, the codimension of $\mathfrak{Z} \cap (\mathfrak{m} \times F)$ in $\mathfrak{m} \times F$ is $\geq n$. This is equivalent to saying that
$$ 
\dim \ (\mathfrak{Z} \cap (\mathfrak{m} \times F)) \leq \dim \ (\mathfrak{m} \times F) - n = \dim \ F  - n <^{\dagger} 0,
$$
where $\dagger$ holds because $F$ is a proper face. But this is a contradiction for a nonempty set. Hence we must have $\mathfrak{Z} \cap (\widehat{X} \times F) = \emptyset$. This proves the lemma.
\end{proof}

We discuss additional algebraic properties of the cycles in $z^n (\Spf (k[[t]]), n)$. Recall:

\begin{lem}\label{lem:module-finite}
Let $(R, \mathfrak{m})$ be a henselian (complete, resp.) local ring and let $R \hookrightarrow B$ be a finite extension of rings.

Then $B$ is a finite direct product of henselian (complete, resp.) local rings. Furthermore, there is a bijection between the factors and the maximal ideals of $B$.

In particular, if $B$ is an integral domain in addition to the above assumptions, then $B$ is a henselian (complete, resp.) local domain with a unique maximal ideal.
\end{lem}

\begin{proof}
See EGA ${\rm IV}_4$ \cite[Propositions (18.5.9), (18,5.10), p.130]{EGA4-4} for the henselian case, and D. Eisenbud \cite[Corollary 7.6, p.190]{Eisenbud} for the complete case.
\end{proof}

\begin{cor}\label{cor:finite free summ}
Let $\mathfrak{Z} \in z^n (\Spf (k[[t]]), n)$ be an integral cycle, and let $B$ be the integral domain such that $\mathfrak{Z}=\Spf (B)$.

Then the canonical homomorphism $k[[t]]\to B$ of rings is finite, and $B$ is a complete local integral domain, which is a free $k[[t]]$-module of finite rank.
\end{cor}

\begin{proof}
By Corollary \ref{cor:finite quasi-finite}, the ring homomorphism $k[[t]]\to B$ is finite. Here $t$ is a non-zerodivisor on $B$ by the special fiber condition (\textbf{SF}) of Definition \ref{defn:HCG} for $\mathfrak{Z}$. Since $B$ is a finitely generated $k[[t]]$-module, where $t$ is not a zero-divisor in $B$, it is torsion free. Thus, by the fundamental theorem of finitely generated modules over a PID, $B$ is a free $k[[t]]$-module of finite rank.

That $B$ is a complete local integral domain follows from Lemma \ref{lem:module-finite}. 
\end{proof}

\begin{defn}\label{defn:free norm}
For an integral cycle $\mathfrak{Z} \in z^n (\Spf (k[[t]]), n)$, let $k[[t]] \to B$ define $\mathfrak{Z}$ as the above. We can define the norm map $N: B^{\times} \to k[[t]]^{\times}$ as follows.

By Corollary \ref{cor:finite free summ}, $B$ is a free $k[[t]]$-module of finite rank. Hence for each $b \in B$, the left multiplication by $b$
$$
L_b: B \to B, \ \ x \mapsto bx 
$$
is a $k[[t]]$-linear endomorphism of the free $k[[t]]$-module $B$ of finite rank. When $b \in B^{\times}$, the map $L_b$ is an automorphism. We define $N(b)$ to be the determinant $\det (L_b) \in k[[t]]^{\times}$. It is independent of the choice of a $k[[t]]$-basis of $B$, and the map $N$ is multiplicative, by standard linear algebra.
\qed
\end{defn}

The finiteness property of $\mathfrak{Z}$ over $k[[t]]$ has another set of consequences. The following is an analogue in the formal setting of the finiteness criterion in \cite[Lemma 2.9]{KP sfs} :

\begin{lem}\label{lem:finiteness}
Let $\mathfrak{Z} \subset \widehat{X} \times \square^n$ be an integral closed formal subscheme. Consider the open immersion $ \widehat{X} \times \square^n \subset \widehat{X} \times (\mathbb{P}^1)^n$. 

Then the following are equivalent:
\begin{enumerate}
\item The morphism $\mathfrak{Z} \to  \widehat{X}$ is finite.
\item The formal scheme $\mathfrak{Z}$ is closed in $\widehat{X} \times (\mathbb{P}^1)^n$.
\end{enumerate}

\end{lem}

\begin{proof}
The proof is almost identical to that of \cite[Lemma 2.9]{KP sfs}, except that we use the corresponding properties for formal schemes.

(1) $\Rightarrow$ (2) : Consider the factorization $ \mathfrak{Z} \hookrightarrow \widehat{X} \times (\mathbb{P}^1)^n \to \widehat{X}$ of the finite morphism. The second morphism is separated of finite type. Thus by Fujiwara-Kato \cite[Corollary 4.6.15, p.340]{FK}, the first morphism is also finite. An immersion is finite only when it is a closed immersion. This implies (2).

(2) $\Rightarrow$ (1): Since $\mathfrak{Z}$ is closed in $\widehat{X} \times (\mathbb{P}^1)^n$, in particular this inclusion morphism is proper. Thus the composite $\mathfrak{Z} \hookrightarrow \widehat{X} \times (\mathbb{P}^1)^n \to \widehat{X}$ is also proper. This implies that the reduction by an ideal of definition of $\widehat{X}$ is also proper (\cite[Proposition 4.7.3, p.341]{FK}). But both $\mathfrak{Z}$ and $ \widehat{X}$ are affine formal schemes, so their reductions are affine schemes. A morphism between affine schemes is proper if and only if it is finite (\cite[Exercise II-4.6, p.106]{Hartshorne}). But this means the morphism $\mathfrak{Z} \to \widehat{X}$ of formal schemes is also finite (\cite[Proposition 4.2.3, p.325]{FK}). This implies (1).
\end{proof}

\begin{cor}\label{cor:no F_1}
Let $\mathfrak{Z} \subset z^n (\widehat{X}, n)$ be an integral cycle. Consider $(\mathbb{P}^1)^n \supset \square^n$, where $y_1, \cdots, y_n$ are the coordinates.

Then it is closed in $\widehat{X} \times (\mathbb{P}^1)^n$ and it does not intersect any divisor of the form $\{ y_i = \epsilon \}$, where $1 \leq i \leq n$ and $\epsilon = 0, \infty, 1 $. 

In particular, for $\mathbb{A}^n \subset (\mathbb{P}^1)^n$ with the coordinates $y_1, \cdots, y_n$, we can regard $\mathfrak{Z}$ as an integral closed formal subscheme of $\widehat{X} \times \mathbb{A}^n$, and there exists a prime ideal $I \subset k[[t]] \{ y_1, \cdots, y_n \}$ such that 
\begin{equation}\label{eqn:no F_1}
\mathfrak{Z} = \Spf ( k[[t]]\{ y_1, \cdots,y_n \}/ I).
\end{equation}
\end{cor}

\begin{proof}
From Corollary \ref{cor:finite quasi-finite}, we know that $\mathfrak{Z}$ is finite over $k[[t]]$, and by Lemma \ref{lem:finiteness}, we know that $\mathfrak{Z}$ is closed in $\widehat{X} \times (\mathbb{P}^1)^n$. Because $\mathfrak{Z} \subset \widehat{X} \times \square^n$ does not intersect the divisors $\{ y_i = 1 \} \subset ( \widehat{X} \times (\mathbb{P}^1)^n) \setminus ( \widehat{X} \times \square^n)$, in fact $\mathfrak{Z} \subset \widehat{X} \times_k \mathbb{A}_k ^n$, and it is closed. This shows \eqref{eqn:no F_1}.

The remaining statement that $\overline{\mathfrak{Z}} \cap \{ y_i = \epsilon \}$ with $\epsilon = 0, \infty$ follows from Lemma \ref{lem:no face}.
\end{proof}

Hence, we no longer need to use the ugly coordinate changes $y_i \mapsto y_i ' = y_i / (y_i -1)$. The generic fibers also enjoy the same property. Namely, Corollary \ref{cor:no F_1} immediately implies:

\begin{cor}\label{cor:no F_1 generic}
Let $\mathfrak{Z} \in z^n (\widehat{X}, n)$ be an integral cycle. Let $\mathfrak{Z}_{\eta}$ be its generic fiber. Then there exists a maximal ideal $\mathfrak{m} \subset k \llp t \rlp \{ y_1, \cdots, y_n \}$ such that $\mathfrak{Z}_{\eta}$ is given by the affinoid algebra $ k \llp t \rlp \{ y_1, \cdots, y_n \}/ \mathfrak{m}$.
\end{cor}

\subsection{Triangular generators}\label{sec:appendix 02}
For a given integral cycle $\mathfrak{Z} \in z^n (\widehat{X}, n)$, in Corollary \ref{cor:pretriangular} we found $n$ polynomial generators in the variables $y_1', \cdots, y_n '$, where $y_i ' = y_i / (y_i -1)$. While it was enough for our purposes of proving finiteness at that time, we now have obtained the finiteness and some of its consequences. This time, we want to improve it further.

The goal of \S \ref{sec:appendix 02} is to find a nice ``triangular" generating set for each integral cycle $\mathfrak{Z} \in z^n (\Spf (k[[t]]), n)$, this time as polynomials in $y_1, \cdots, y_n$. We will use Corollaries \ref{cor:no F_1} and \ref{cor:no F_1 generic} we obtained in \S \ref{sec:appendix 01} as consequences of the finiteness of $\mathfrak{Z}$ over $k[[t]]$.

\begin{prop}\label{prop:triangular}
Let $\mathfrak{Z} \in z^n (\Spf (k[[t]]), n)$ be an integral cycle. 

Then there are \textbf{polynomials} in $k[[t]] [y_1, \cdots, y_n ] $ of the form
\begin{equation}\label{eqn:triangular}
\tuborg
P_1 (y_1) \in k[[t]] [ y_1 ],\\
\vdots \\
P_n (y_1, \cdots, y_n ) \in k[[t]] [ y_1, \cdots, y_n ],\sluttuborg
\end{equation}
such that 
\begin{enumerate}
\item the ideal $(P_1 (y_1), \cdots, P_n (y_1, \cdots, y_n))$ in $k[[t]] \{ y_1, \cdots, y_n \}$ defines $\mathfrak{Z}$,
\item for each $i$, the polynomial $P_i$ is monic in $y_i$. When $i=1$, the constant term of $P_1$ is in $k[[t]]^{\times}$, while for $2 \leq i \leq n$ and for the image $\bar{P}_i (y_i)$ of $P_i$ in the quotient ring $\left( \frac{k[[t]]\{y_1, \cdots, y_{i-1} \}}{(P_1, \cdots, P_{i-1})}\right) [ y_i]$, its constant term is a unit in $\left( \frac{k[[t]]\{y_1, \cdots, y_{i-1} \}}{(P_1, \cdots, P_{i-1})}\right)^{\times}$, and
\item for $2 \leq i \leq n$, for $P_i$ regarded as a polynomial in $y_i$, its coefficients in $ k[[t]] [ y_1, \cdots, y_{i-1}]$ have their $y_j$-degrees strictly less than $\deg_{y_j} P_j$ for all $1 \leq j < i$.
\end{enumerate}
\end{prop}

\begin{proof}
Let $\eta \in \Spec (k[[t]])$ be the generic point. By Corollary \ref{cor:no F_1 generic}, the generic fiber $\mathfrak{Z} _{\eta}$ is given by a maximal ideal $\mathfrak{m} \subset T_n=  k\llp t\rlp  \{ y_1, \cdots, y_n \}$. 

Applying Lemma \ref{lem:Bosch 17} to $\mathfrak{m}$ with $z_i= y_i$ for $1 \leq i \leq n$, we obtain a triangular shaped set of polynomials $p_1 (y_1) \in k[[t]][y_1],  p_2 (y_1, y_2) \in k \llp t \rlp [y_1, y_2], \cdots, p_n (y_1, \cdots, y_n) \in k \llp t \rlp [y_1, \cdots, y_n]$ that generate the maximal ideal $\mathfrak{m} \subset T_n$, and satisfy the properties there in the lemma.

 To obtain the generators for its closure $\mathfrak{Z}$, for each $1 \leq i \leq n$ we clear the denominators of the coefficients in $k \llp t \rlp$ of the terms of $p_i$ by the l.c.m.~of them in $k[[t]]$. We can do it because $k[[t]]$ is a UFD. So obtained elements are denoted by $P_1 (y_1) \in k[[t]] [y_1], \cdots, P_n  (y_1, \cdots, y_n) \in k[[t]] [ y_1, \cdots, y_n ]$, and they satisfy the properties that
 \begin{enumerate}
\item [(a)] the ideal $(P_1, \cdots, P_n)$ generated in $k[[t]]\{y_1, \cdots, y_n\}$ defines $\mathfrak{Z}$,
\item [(b)] for each $i$, the highest $y_i$-degree term of $P_i$ involves no other variable, and
\item [(c)] for $2 \leq i \leq n$, when $P_i$ is regarded as a polynomial in $y_i$, each of its coefficients in $k[[t]][y_1, \cdots, y_{i-1}]$ has its $y_{j}$-degree strictly less than $\deg_{y_j} P_j$ for all $ 1 \leq j < i$.
 \end{enumerate}

The conditions (1) and (3) of the proposition are satisfied by the above properties (a) and (c), but the condition (2) is not yet achieved. We show that (2) can be achieved for $\mathfrak{Z}$. 

\bigskip

For $1 \leq i \leq n$, write $S_i:= k[[t]] \{ y_1, \cdots, y_i \}$. Let $\mathfrak{p}_n \subset S_n$ be the prime ideal that defines $\mathfrak{Z}$ (by Corollary \ref{cor:no F_1}). For the sequence of injective ring homomorphisms
$$
k[[t]] \hookrightarrow S_1 \hookrightarrow \cdots \hookrightarrow S_{n-1} \hookrightarrow S_n,
$$
consider the prime ideals $\mathfrak{p}_i:= \mathfrak{p}_n \cap S_i$ for $1 \leq i \leq n-1$. By construction, $\mathfrak{p}_i = (P_1, \cdots, P_i)$, and 
$$
k[[t]] \hookrightarrow S_1/ \mathfrak{p}_1 \hookrightarrow \cdots \hookrightarrow S_{n-1}/\mathfrak{p}_{n-1} \hookrightarrow S_n/\mathfrak{p}_n.
$$

Since $k[[t]] \to S_n/ \mathfrak{p}_n$ is finite and free (Corollary \ref{cor:finite free summ}), and $k[[t]]$ is a PID, all homomorphism $k[[t]] \to S_i / \mathfrak{p}_i$ for $1 \leq i \leq n$ are also finite and free because a submodule of a free module of finite rank over a PID is again free by the fundamental theorem on finitely generated modules over a PID. 

We let $\mathfrak{Z} ^{(i)} \subset \Spf (k[[t]]) \times \square^i$ be the integral closed formal subscheme given by $\Spf (S_i / \mathfrak{p}_i)$. By construction, its dimension is equal to that of $\mathfrak{Z}$, so that its codimension is $i$. This $\mathfrak{Z}^{(i)}$ is equal to the image of $\mathfrak{Z}$ under the projection $pr_i: \Spf (k[[t]]) \times \square^n \to \Spf (k[[t]]) \times \square^i $, that sends $(y_1, \cdots, y_n) \mapsto (y_1, \cdots, y_i)$. 

\medskip

\textbf{Claim:} \emph{$\mathfrak{Z}^{(i)} \in z^i (\Spf (k[[t]]), i)$ for $1 \leq i \leq n$.}

\medskip

For each codimension $1$ face $F= F_j ^{\epsilon} \subset \square_k ^n$ with $j \leq i$, the intersection $\mathfrak{Z}^{(i)} \cap (\widehat{X} \times F)$ is the image of $\mathfrak{Z} \cap (\widehat{X} \times F)$ under $pr_i$. But the latter is empty by Lemma \ref{lem:no face}, thus so is the former. This means $\mathfrak{Z}^{(i)}$ has the empty intersection with all proper faces. In particular, the general position condition (\textbf{GP}) of Definition \ref{defn:HCG} for $\mathfrak{Z}^{(i)}$ holds trivially.

Since $\mathfrak{Z}^{(i)}$ has the empty intersection with all proper faces, to check the condition (\textbf{SF}) for $\mathfrak{Z}^{(i)}$, it remains to take $F= \square_k ^n$, and we need to look at its special fiber. Since $\mathfrak{Z}^{(i)}$ is finite over $k[[t]]$ and the special fiber is given by its unique closed point by Lemma \ref{lem:module-finite}, the condition (\textbf{SF}) holds for $\mathfrak{Z}^{(i)}$. 
Thus $\mathfrak{Z}^{(i)} \in z^i (\Spf (k[[t]]), i)$.

\bigskip

We prove the remaining part (2) of the proposition.

Suppose $i=1$. Then the intersection of $\mathfrak{Z}^{(1)}$ with the faces $\{ y_1 = \infty \}$ and $\{ y_1 = 0 \}$ are empty by Lemma \ref{lem:no face} and Claim. Thus the leading coefficient and the constant term of $P_1 (y_1)$ in $y_1$ are both units in $k[[t]]^{\times}$. After scaling by the leading coefficient of $P_1 (y_1)$ in $y_1$, the leading coefficient becomes $1$, and the constant term is still a unit in $k[[t]]^{\times}$, proving (2) for $i=1$.

Let $i  \geq 2$. Among the defining polynomials $P_1, \cdots, P_i$ of $\mathfrak{Z}^{(i)}$, the only one that involves the variable $y_i$ is $P_i$. Furthermore, the highest $y_i$-degree term of $P_i$ does not involve any other variables. Hence the empty intersection of $\mathfrak{Z}^{(i)}$ with the codimension $1$ face $\{ y_i = \infty\}$ (Lemma \ref{lem:no face} and the above Claim) means that the leading coefficient of $P_i$ in $y_i$ is a unit in $k[[t]]^{\times}$. Hence after scaling by the leading coefficient of $P_i$ in $y_i$, the leading coefficient becomes $1$, so that we may assume $P_i$ is monic in $y_i$. 

Note that scaling by units of $k[[t]]^{\times}$ does not disturb the pre-existing properties (1), (3).

On the other hand, that the intersection of $\mathfrak{Z}^{(i)}$ with $\{y_i = 0 \}$ is empty means that the image $\bar{P}_i$ of $P_i$ in the quotient ring $\left( \frac{k[[t]]\{y_1, \cdots, y_{i-1} \}}{(P_1, \cdots, P_{i-1})}\right) [ y_i]$ has a unit constant term, when regarded as a polynomial in $y_i$. 

 Thus we have achieved all of (1), (2), (3). 
\end{proof}

\section{The graph homomorphisms and the regulators}\label{sec:Milnor}
 
\subsection{Milnor $K$-theory and the graph homomorphism}\label{sec:Milnor2}

We discuss the graph homomorphisms from the Milnor $K$-groups. Recall from Elbaz-Vincent--M\"uller-Stach \cite[Lemma 2.1]{EVMS} that for an equidimensional $k$-algebra $R$ essentially of finite type, there is the graph homomorphism
$$
gr:K_n ^M (R) \to \CH^n (\Spec (R), n).
$$
We have a similar homomorphism for the higher Chow groups of some affine formal schemes. To avoid unnecessary complexities, we consider the relevant case only:

\begin{lem}\label{lem:gr map}
Consider $\widehat{X}= \Spf (k[[t]])$. For each Milnor symbol $\{ a_1, \cdots, a_n \} \in K_n ^M (k[[t]])$, with $a_i \in k[[t]]^{\times}$, consider the integral closed formal subscheme $\Gamma_{(a_1, \cdots, a_n )} \subset \square_{\widehat{X}} ^n$ defined by the system of linear polynomials
$$
  y_1 - a_1,  \ \ \cdots,  \ \ y_n - a_n.
 $$

Then sending $\{a_1, \cdots, a_n \}$ to the graph cycle $\Gamma_{(a_1, \cdots, a_n)}$ defines the graph homomorphism
\begin{equation}\label{eqn:Milnor0}
gr: K_n ^M (k[[t]]) \to \CH^n (\Spf (k[[t]]), n).
\end{equation}
\end{lem}

\begin{proof}
One checks immediately that $\Gamma_{(a_1, \cdots, a_n) } \in z^n (\Spf (k[[t]]), n)$. By Lemma \ref{lem:no face}, this represents a class in $\CH^n (\Spf (k[[t]]), n)$. 

The rest of the argument is essentially a repetition of \cite[Lemma 2.1]{EVMS}. We remark that \emph{loc.cit.}~makes a general assumption that their ring is essentially of finite type, but this is not used there, so we can still follow a large part of the argument. 

Firstly, for $a, a_i \in k[[t]]^{\times}$ such that $1-a \in k[[t]]^{\times}$ for $3 \leq i \leq n$, consider the parametrized cycle given by the closure of the graph of the rational map
$$
\mathfrak{W}_1: \square_{\widehat{X}} ^1 \dashedrightarrow \square_{\widehat{X}} ^{n+1}, \ \ x \mapsto \left( x, 1-x, \frac{ a-x}{1-x}, a_3, \cdots, a_n \right).
$$
One checks $\mathfrak{W}_1 \in z^n (\Spf (k[[t]]), n+1)$. By straightforward calculations, one sees that the only nonzero face of $\mathfrak{W}_1$ is 
$$
\partial_3 ^0 \mathfrak{W}_1 = \Gamma_{(a, 1-a, a_3, \cdots, a_n)}.
$$
In particular, the Steinberg symbol $\{ a, 1-a, a_3, \cdots, a_n \}\in K_n ^M (k[[t]])$ maps to $0$ in $\CH^n (\Spf (k[[t]]), n)$.

Secondly, for $a, b, a_i \in k[[t]]^{\times}$ for $2 \leq i \leq n$, consider the parametrized cycle
$$
\mathfrak{W}_2: \square_{\widehat{X}} ^1 \dashedrightarrow \square_{\widehat{X}} ^{n+1}, \ \ x \mapsto \left( x, \frac{ ax - ab}{x-ab}, a_2, \cdots, a_n \right).
$$
One checks $\mathfrak{W}_2 \in z^n (\Spf (k[[t]]), n+1)$, and that its only nonzero faces are
$$
\tuborg
\partial_1 ^{\infty} \mathfrak{W}_2 = \Gamma_{(a, a_2, \cdots, a_n)}, \\
\partial_2 ^0 \mathfrak{W}_2 = \Gamma_{(b, a_2, \cdots, a_n)} , \\
\partial_2 ^{\infty} \mathfrak{W}_2 = \Gamma_{(ab, a_2, \cdots, a_n)}.\sluttuborg
$$
Hence killing $\partial \mathfrak{W}_2$ in $\CH^n (\Spf (k[[t]]), n)$, we deduce
$$
\Gamma_{(ab, a_2, \cdots, a_n)} \equiv \Gamma_{(a, a_2, \cdots, a_n)} + \Gamma_{(b, a_2, \cdots, a_n)} \ \ \mbox{ in } \CH^n (\Spf (k[[t]]), n).
$$

Considering various permutations of the cycle $ \mathfrak{W}_2$, we thus deduce
$$
\bigotimes_{i=1} ^n k[[t]]^{\times} \to \CH^n (\Spf (k[[t]]), n),
$$
and then using the permutations of $\mathfrak{W}_1$, it further descends to \eqref{eqn:Milnor0}.
\end{proof}

\begin{lem}\label{lem:Milnor--0}
Let $m \geq 2$. 
The graph map \eqref{eqn:Milnor0} induces the  homomorphism
\begin{equation}\label{eqn:Milnor-1}
gr_{k_m}: K_n ^M (k_m) \to \BGH^n (\Spec (k_m), n),
\end{equation}
where the codomain is the new higher Chow group defined in Definition \ref{defn:defn Milnor}. 
\end{lem}

\begin{proof}
Since $k[[t]]$ is a local ring, the natural homomorphism $k[[t]]^{\times} \to k_m ^{\times}$ is surjective (see e.g. Hartshorne-Polini \cite[Lemma 5.2]{HP}). Thus $K_n ^M (k[[t]]) \to K_n ^M (k_m)$ is also surjective. Hence to see that \eqref{eqn:Milnor0} descends to \eqref{eqn:Milnor-1}, it remains to check that $\ker (K_n ^M (k[[t]]) \to K_n ^M (k_m))$ is mapped to $0$ under the graph map \eqref{eqn:Milnor0} composed with the natural map $\CH^n (\Spf (k[[t]]), n) \to \BGH^n (\Spec (k_m), n)$. 

Since the above kernel is generated by elements of the form $\{ 1 + t^m c, c_2, \cdots, c_n \}$, where $c\in k[[t]]$ and $c_i \in k[[t]]^{\times}$ for $2 \leq i \leq n$ (see Kato-Saito \cite[Lemma 1.3.1, p.261]{KS}), it is enough to show that 
\begin{equation}\label{eqn:KS graph}
\Gamma_{ ( 1 + t^m c, c_2, \cdots, c_n )} \equiv 0 \ \ \mbox{ in } \BGH^n (\Spec (k_m), n),
\end{equation}
where $\Gamma_{ ( 1 + t^m c, c_2, \cdots, c_n )} $ is the graph cycle in $z^n (\widehat{X}, n)$ of the above Milnor symbol, defined in Lemma \ref{lem:gr map}.

\medskip

Before we do so, we first prove:

\textbf{Claim:}\emph{
Let $a_i, b_i \in k[[t]]^{\times}$ such that $a_i - b_i \in (t^m)k[[t]]$ for $1 \leq i \leq n$. Then have 
$$\Gamma_{a} - \Gamma_b \in \mathcal{M}^n (\widehat{X}, Y, n),$$
where $\Gamma_a:= \Gamma_{(a_1, \cdots, a_n)}$ and $\Gamma_b:= \Gamma_{(b_1, \cdots, b_n)} $ denote the corresponding integral graph cycles in $z^n (\widehat{X}, n)$. In particular, $\Gamma_a \equiv \Gamma_b$ in $ \BGH^n (\Spec (k_m), n)$.
}

\medskip

 This claim is immediate because we have an isomorphism
$$
\mathcal{O}_{\Gamma_a} \otimes_{\mathcal{O}_{\square_{\widehat{X}} ^n}} \mathcal{O}_{\square_Y^n} \simeq 
\mathcal{O}_{\Gamma_b} \otimes_{\mathcal{O}_{\square_{\widehat{X}} ^n}} \mathcal{O}_{\square_Y^n}
$$
of $\mathcal{O}_{\square_Y^n}$-algebras.

\medskip

Going back to the proof of \eqref{eqn:KS graph}, let $d_i:= \bar{c}_i \in k^{\times}$ for $2 \leq i \leq n$. Via the obvious injective homomorphism $k^{\times} \hookrightarrow k[[t]]^{\times}$, regard $d_i \in k[[t]]^{\times}$. Then we have
$$
1 + t^m c \equiv 1, \  \ c_2 \equiv d_2 , \  \ \cdots ,  \ \ c_n \equiv d_n \ \mod (t^m) k[[t]].
$$
Thus by the Claim, we have
\begin{equation}\label{eqn:KS graph-1}
\Gamma_{ ( 1 + t^m c, c_2, \cdots, c_n )} \equiv \Gamma_{(1, d_2, \cdots, d_n ) } \ \mbox{ in } \BGH^n (\Spec (k_m), n).
\end{equation}
However, $\square= \mathbb{P}^1 \setminus \{1 \}$ so that $\Gamma_{ (1, d_2, \cdots, d_n) } = \emptyset$ in $\square_{\widehat{X}} ^n$. Thus \eqref{eqn:KS graph-1} proves \eqref{eqn:KS graph}. This induces \eqref{eqn:Milnor-1}.\end{proof}

\bigskip

The goal of this article, Theorem \ref{thm:main intro}, is to prove that \eqref{eqn:Milnor-1} is an isomorphism when $|k|\gg 0$. It will be completed in Theorem \ref{thm:k_m main}.

\subsection{Regulators via rigid analytic spaces}\label{sec:suslin}
In \S \ref{sec:suslin}, we construct a homomorphism
\begin{equation}\label{eqn:cycle K regulator}
 \widehat{\Upsilon}: \CH^n (\Spf (k[[t]]), n) \to K_n ^M ( k\llp t\rlp ).
 \end{equation}

Under the additional assumption that $|k| \gg 0$, the homomorphism $K_n ^M (k[[t]]) \to K_n ^M (k\llp t\rlp )$ is injective (Gersten conjecture for $K_n ^M$ proven by M. Kerz \cite[Proposition 10, p.181]{Kerz finite}). We will see that the image of $\widehat{\Upsilon}$ lies in $K_n ^M (k[[t]])$.

\bigskip

To facilitate our discussions, we introduce a rigid analytic analogue of higher Chow groups in the Milnor range (see also J. Ayoub \cite[Notations 2.4.14, p.297 and beyond]{Ayoub}):

\begin{defn}\label{defn:HCG rigid}
Let $K$ be a complete discrete valued field with a nontrivial non-archimedean norm. Let $K^{\circ}$ be its valuation ring and suppose $k \subset K^{\circ}$. We define the rigid analytic analogue of higher Chow groups for $\Sp (K)$ for the Milnor range, but its generalization to the off-Milnor ranges is similar.

We consider the rigid analytic space $\Sp (K) \times \square_k ^n$. 
(The fiber product exists by Bosch-G\"untzer-Remmert \cite[\S 7.1.4, Proposition 4, p.268]{BGR}.) We define $z^n (\Sp (K), n)$ to be the free abelian group on the codimension $n$ points not lying on $\{y_i = 0, 1, \infty\}$ for $1 \leq i \leq n$, and define $\un{z}^n (\Sp (K) , n+1)$ be the free abelian group on the codimension $n$ integral closed subsets that intersect all faces of $\Sp (K) \times \square_k ^{n+1}$ properly, where the faces are similarly defined as in \S \ref{sec:HC}. We define $z^n (\Sp (K), *)$ for $*=n, n+1$ by modding out $\un{z}^n (\Sp (K), *)$ by the degenerate cycles. Via the closed immersions $\iota_i ^{\epsilon}: \Sp (K) \times \square_k ^{n} \hookrightarrow \Sp (K) \times \square_k ^{n+1}$ for $1 \leq i \leq n+1$ and $\epsilon = 0, \infty$, we have the natural boundary map $\partial: {z}^n (\Sp (K), n+1) \to z ^n (\Sp (K), n)$ defined by $\partial = \sum_{i=1} ^{n+1} (-1)^i ( \partial_i ^{\infty} - \partial_i ^0)$, where $\partial_i ^{\epsilon} = (\iota_i ^{\epsilon})^*$. Its cokernel is denoted by $\CH ^n (\Sp (K), n)$. \qed
\end{defn}

\begin{exm}\label{exm:loc}
If we take $K= k\llp t\rlp $ and an integral cycle $\mathfrak{Z} \in \CH^n (\Spf (k[[t]]),n)$, then the generic fiber $\mathfrak{Z}_{\eta}$ gives a member of ${\CH} ^n ( \Sp (k\llp t\rlp ), n)$. With Corollaries \ref{cor:no F_1} and \ref{cor:no F_1 generic} in mind, note that the localization homomorphisms $k[[t]] \{ y_1, \cdots, y_n \} \to k\llp t\rlp \{ y_1, \cdots, y_n \}$ over $n \geq 1$ induce the commutative diagram
$$
\xymatrix{
z^n (\Spf (k[[t]]), n+1) \ar[d] ^{\partial} \ar[r] & z^n ({\rm Sp} (k\llp t\rlp ), n+1) \ar[d] ^{\partial} \\
z^n (\Spf (k[[t]]), n) \ar[r] & z^n ({\rm Sp} (k\llp t\rlp ), n),}
$$
from which we deduce $\eta: \CH^n (\Spf (k[[t]]), n) \to \CH^n ( \Sp  (k\llp t\rlp ), n) $. 
\qed
\end{exm}

What follows below is essentially an argument in B. Totaro \cite[\S 3]{Totaro}, with some suitable modifications to fit it into the rigid analytic situation being considered. 

\begin{defn}\label{defn:243}
Let $\mathfrak{p} \in z^n (\Sp ( k\llp t\rlp ), n)$ be a point. It gives a maximal ideal $\mathfrak{m}$ of $T_n= k \llp t \rlp \{ y_1, \cdots, y_n \}$ (cf. Corollary \ref{cor:no F_1 generic}). By Lemma \ref{lem:Tate}-(3), the injection  $  k\llp t\rlp  \hookrightarrow L:=T_n / \mathfrak{m}$ is a finite extension of fields. This point $\mathfrak{p}$ defines an $n$-tuple $(x_1, \cdots, x_n)$ whose values lie in $L \setminus \{ 0, 1 \}$, namely $x_i = \bar{y}_i$ in $T_n / \mathfrak{m} = L$. We consider the Milnor symbol $ \{ x_1, \cdots, x_n\} \in K_n ^M (L)$.

 Define 
\begin{equation}\label{eqn:norm upsilon}
\tilde{\Upsilon} (\mathfrak{p}):= N_{L/ k\llp t\rlp } \{ x_1, \cdots, x_n\} \in K_n ^M (k\llp t\rlp ),
\end{equation}
via the norm $N_{L/ k\llp t\rlp }: K_n ^M (L) \to K_n ^M (k\llp t\rlp )$ of Bass-Tate \cite{BassTate} and K. Kato \cite{Kato}. \qed
\end{defn}

\begin{lem}\label{lem:5.1.4}
The composite
$$
z^n (\Sp (k\llp t\rlp ), n+1) \overset{\partial}{\to} z^n (\Sp (k\llp t\rlp ), n) \overset{\tilde{\Upsilon}}{\to} K_n ^M (k\llp t\rlp )
$$
is zero. In particular, we have the induced homomorphism
\begin{equation}\label{eqn:hat CH upsilon}
\tilde{\Upsilon} : {\CH} ^n (\Sp (k\llp t\rlp ), n) \to K_n ^M (k\llp t\rlp ).
\end{equation}
\end{lem}

\begin{proof}

We continue to follow B. Totaro \cite[\S 3]{Totaro} with suitable modifications.

Let $C \in z ^n (\Sp (k\llp t\rlp ), n+1)$ be an integral cycle. Take the normalization $D \to C$; since $C$ is defined by a $1$-dimensional noetherian integral domain, it has a unique normalization (see B. Conrad \cite{Conrad cpnt} for normalizations of more general rigid analytic spaces). The normalization map is finite because every affinoid $k\llp t \rlp$-algebra that is an integral domain is a Japanese ring (see Bosch-G\"untzer-Remmert \cite[\S 6.1.2, Proposition 4, p.228]{BGR} or EGA ${\rm IV}_2$ \cite[Scholie (7.8.3)-(vi), p.215]{EGA4-2}). As in \cite[\S 3]{Totaro}, using W. Fulton \cite[Example 1.2.3, p.9]{Fulton}, it is enough to show that the boundary (defined by pull-back of the faces and push-forward) of $D$ maps to $0$ in $K_n ^M (k \llp t \rlp)$.

Let $\overline{D}$ be the unique regular compactification of $D$, which exists as $\dim \ D  =1$. By the rigid analytic GAGA (see e.g. \cite[Theorem 4.10.5, p.113]{FvdP}), this $\overline{D}$ can be regarded as a regular projective algebraic curve over $k \llp t \rlp$. Let $\mathbb{K}$ be the rational function field of $\overline{D}$, which is equal to that of $C$ and $D$. Since $\mathbb{K} \supset k \llp t \rlp$, it is an infinite field.

The map $D \to \Sp (k\llp t\rlp ) \times \square_k ^{n+1}$ is given by $(n+1)$ rational functions $g_1, \cdots, g_{n+1} \in \mathbb{K}$ on $D$. By the given proper intersection conditions with the faces, no $g_i$ is identically equal to $0$ or $\infty$, and any closed point $w \in D$ with $g_i (w) = 0$ or $\infty$ has $g_j (w) \not \in \{ 0, \infty \}$ for all $j \not = i$.

 For each closed point $w \in \overline{D}$, there is its associated discrete valuation on $\mathbb{K}$ with the associated boundary map $\partial_w: K_{n+1} ^M (\mathbb{K}) \to K_n ^M (\kappa (w))$ (see J. Milnor \cite[Lemma 2.1]{Milnor IM}), where $\kappa (w)$ is the residue field of $w$.

Since $\mathbb{K}$ is infinite, by Suslin reciprocity (see \cite{Suslin}), for the Milnor symbol $\{ g_1, \cdots, g_{n+1} \} \in K_{n+1} ^M (\mathbb{K})$, we have 
$$\sum_{w \in \overline{D}} N_{\kappa (w)/ k\llp t\rlp } \partial _w \{ g_1, \cdots, g_{n+1} \} = 0 \ \ \mbox{ in } K_n ^M (k\llp t\rlp ).$$

If $w \in \overline{D} \setminus D$, then one of $g_i$ has $g_i (w) = 1$ so that $\partial_w \{ g_1, \cdots, g_{n+1} \} = 0$. Hence the above sum can be written without such points, and 
\begin{equation}\label{eqn:Suslin0}
\sum_{w \in D} N_{\kappa (w)/ k\llp t\rlp } \partial _w \{ g_1, \cdots, g_{n+1} \} = 0 \ \ \mbox{ in } K_n ^M (k\llp t\rlp ).
\end{equation}
On the other hand, by following the definition of $\partial_w$, we have the commutative diagram
$$
\xymatrix{
z ^{n} (\Sp (k\llp t\rlp ), n+1) \ar[rr] ^{\partial}  \ar[d] & & z ^n (\Sp (k\llp t\rlp ), n) \ar[d] ^{\tilde{\Upsilon}} \\
K_{n+1} ^M (\mathbb{K}) \ar[rr] ^{\sum_{w} N_{\kappa (w)/k\llp t\rlp } \partial_w} & & K_n ^M (k \llp t \rlp ),}
$$
where the left vertical arrow associates $C$ to the Milnor symbol $\{ g_1, \cdots, g_{n+1} \}$ in $K_{n+1} ^M (\mathbb{K})$. Thus $\tilde{\Upsilon} \circ \partial C $ is equal to \eqref{eqn:Suslin0}, which is $0$. Thus, we deduce the map \eqref{eqn:hat CH upsilon} as desired.
\end{proof}

\begin{defn}
We define the map $\widehat{\Upsilon}$ as the composite
$$
\widehat{\Upsilon}: \CH^n (\Spf (k[[t]]), n) \overset{\eta}{\to} \CH^n ( \Sp  (k\llp t\rlp ), n) \overset{\tilde{\Upsilon}}{\to} K_n ^M (  k\llp t\rlp ),
$$
where $\eta$ is the localization map (Example \ref{exm:loc}) and $\tilde{\Upsilon}$ is from Definition \ref{defn:243} and Lemma \ref{lem:5.1.4}. \qed
\end{defn}

\begin{lem}\label{lem:inj k[[t]]}
Let $k$ be a field with $|k| \gg 0$ so that $K_n ^M (k[[t]]) = \widehat{K}_n ^M (k[[t]])$.

Then the image of the composite 
$$
K_n ^M (k[[t]]) \overset{gr}{\to} \CH^n (\Spf (k[[t]]), n) \overset{\widehat{\Upsilon}}{ \to} K_n ^M (k \llp t\rlp ),
$$
lands into the {subgroup} $K_n ^M (k[[t]]) \subset K_n ^M (k\llp t\rlp )$, and the composite is the identity on $K_n ^M (k[[t]])$.

In particular, the map $gr$ is injective and the image of $\CH^n (\Spf (k[[t]]), n)$ under $\widehat{\Upsilon}$ is in the subgroup $K_n ^M (k[[t]])$ of $K_n ^M (k\llp t\rlp )$.
\end{lem}

Before the proof, we remark that the condition $K_n ^M (k[[t]]) = \widehat{K}_n^M  (k[[t]])$ implies that the natural homomorphism $K_n ^M (k[[t]]) \to K_n ^M (k\llp t\rlp )$ is injective (see M. Kerz \cite[Proposition 10, p.181]{Kerz finite}), thus we can regard $K_n ^M (k[[t]]) \subset K_n ^M (k\llp t\rlp )$.

When $n=1$, this always holds for any field $k$.

\begin{proof}
Consider the Milnor symbol $\{ a_1, \cdots, a_n \} \in K_n ^M (k[[t]])$ with $a_i \in k[[t]]^{\times}$. Its image under the first map $gr$ is given by the integral graph cycle $\mathfrak{Z}$ defined by the system of polynomials
$$
 y_1 - a_1, \ \  \cdots, \ \  y_n - a_n.
$$
Its coordinates are given by $a_1, \cdots, a_n \in k[[t]]^{\times} \subset k\llp t\rlp ^{\times}$, and they are of degree $1$ over $k\llp t\rlp $. Hence the generic fiber $\mathfrak{Z}_{\eta}$ gives $L= k\llp t\rlp $ in the notations of Definition \ref{defn:243}, and the norm map in \eqref{eqn:norm upsilon} is the identity map. Hence $\tilde{\Upsilon} (\mathfrak{Z}) = \{ a_1, \dots, a_n \}$ and it belongs to the subgroup $K_n ^M (k[[t]])$. This shows that the composite is the identity map of $K_n ^M (k[[t]])$. In particular, $gr$ is injective.

Since the image of the composite lies in $K_n ^M (k[[t]])$, so does the image of the group $\CH^n (\Spf (k[[t]]), n)$.
\end{proof}

The following is a rigid analytic analogue of \cite{Totaro}:

\begin{cor}\label{cor:k((t)) iso}
For any field $k$, the composite
$$
K_n ^M (k\llp t\rlp )  \overset{gr}{\to}  \CH^n (\Sp (k\llp t\rlp ), n) \overset{\tilde{\Upsilon}}{ \to} K_n ^M (k\llp t\rlp )
$$ 
is the identity map, and all arrows are isomorphisms.
\end{cor}
\begin{proof}
For the injectivity of $gr$, we simply note that the composite $\tilde{\Upsilon} \circ gr$ evaluated at the Milnor symbols $\{ a_1, \cdots, a_n \} \in K_n ^M (k\llp t\rlp )$ for some $a_i \in (k\llp t\rlp ) ^{\times}$ is again $\{ a_1, \cdots, a_n \} $. Thus $\tilde{\Upsilon} \circ gr = {\rm Id}$. In particular, $gr$ is injective.

For the surjectivity of $gr$, we may follow B. Totaro \cite[\S 4]{Totaro}. When $F$ is a field, to prove that the graph homomorphism $K_n ^M (F) \to \CH^n (\Spec (F), n)$ to the ordinary cubical higher Chow group, is surjective, Totaro used the norm $N : L ^{\times} \to F ^{\times}$ to relate the class of a given closed point to a cycle given by $F$-rational points, where $L$ is the residue field of the point.

Specializing to the case $F= k\llp t\rlp$ (which is infinite), we can repeat the argument of Totaro \emph{mutatis mutandis}. As this corollary does not play an important role for the rest of the article, we omit details.
\end{proof}

\subsection{A comparison of equivalence relations}\label{sec:mod Y mod I}
We continue to use the notations $Y= \Spec (k_m)$ and $\widehat{X}= \Spf (k[[t]])$. The closed immersion $Y \hookrightarrow \widehat{X}$ is given by the ideal $I:= (t^m) \subset k[[t]]$ so that $k[[t]]/ I = k_m$.

Recall from Remark \ref{remk:mod Y=naive Y} that, we have the mod $Y$-equivalence $\sim_Y$ on $z^n (\widehat{X}, n)$ given by the subgroup $\mathcal{M}^n (\widehat{X}, Y, n)$, and we have another equivalence $\sim_I$ given by the subgroup $\mathcal{N}^n (\widehat{X}, Y, n)$ generated by the differences of a pair of integral cycles that are equal after taking the mod $I$-reductions. We prove in Lemma \ref{lem:mod Y=naive Y} below that these two are equivalent in the Milnor range. It will simplify our discussions from \S \ref{sec:strong graph}.

\begin{lem}\label{lem:mod Y=naive Y}
We have the equality
\begin{equation}\label{eqn:M=N}
\mathcal{M}^n (\widehat{X}, Y, n) = \mathcal{N} ^n (\widehat{X}, Y, n)
\end{equation}
of the subgroups in $z^n (\widehat{X}, n)$.
\end{lem}

\begin{proof}
For each generator $[\mathfrak{Z}_1]- [\mathfrak{Z}_2] \in \mathcal{N}^n (\widehat{X}, Y, n)$ associated to a pair of integral cycles $\mathfrak{Z}_1$ and $\mathfrak{Z}_2 \in z^n (\widehat{X}, n)$, by definition we have
$$
\mathcal{O}_{\mathfrak{Z}_1} \otimes _{\mathcal{O}_{\square_{\widehat{X}} ^n} } \mathcal{O}_{\square_Y^n} \simeq 
\mathcal{O}_{\mathfrak{Z}_2} \otimes _{\mathcal{O}_{\square_{\widehat{X}} ^n} } \mathcal{O}_{\square_Y^n}
$$
as $\mathcal{O}_{\square_{Y}^n}$-algebras. Hence $[\mathfrak{Z}_1] - [\mathfrak{Z}_2] \in \mathcal{M}^n (\widehat{X}, Y, n)$, and this shows $\mathcal{M}^n (\widehat{X}, Y, n) \supset \mathcal{N} ^n (\widehat{X}, Y, n)$. It remains to prove the opposite inclusion.

\medskip

Let $(\mathcal{A}_1, \mathcal{A}_2) \in \mathcal{L}^n (\widehat{X}, Y, n)$. We have an isomorphism of $\mathcal{O}_{\square_Y^n}$-algebras
\begin{equation}\label{eqn:mod naive compare 1}
\mathcal{A}_1 \otimes _{\mathcal{O}_{\square_{\widehat{X}}^n} } \mathcal{O}_{\square_Y^n} \simeq \mathcal{A}_2 \otimes _{\mathcal{O}_{\square_{\widehat{X}}^n} } \mathcal{O}_{\square_Y^n}.
\end{equation}

Let $B_j$ be the ring of the global sections of $\mathcal{A}_j$. Since $[\mathcal{A}_j] \in z^n (\widehat{X}, n)$, by Corollary \ref{cor:finite quasi-finite} we see that $B_j$ is a $k[[t]]$-algebra which is a finite $k[[t]]$-module. In particular, it is a semi-local $k[[t]]$-algebra.

Having Corollary \ref{cor:no F_1} in mind, the $k_m[y_1, \cdots, y_n]$-algebra isomorphism \eqref{eqn:mod naive compare 1} in particular implies that we have an isomorphism of Artin rings
\begin{equation}\label{eqn:artin triangle}
B_1 / (t^m)B_1 \simeq B_2 / (t^m) B_2.
\end{equation}
By the structure theorem for Artin rings (Atiyah-MacDonald \cite[Theorem 8.7, p.90]{AM} or D. Eisenbud \cite[Corollary 2.16, p.76]{Eisenbud}), each Artin ring decomposes uniquely into a product of Artin local rings, up to isomorphism. So, both sides of \eqref{eqn:artin triangle} decompose uniquely into products of Artin local rings, and under the isomorphism, there is a 1-1 correspondence between the Artin local factors of both sides.

On the other hand, since $k[[t]]$ is henselian and complete, by Lemma \ref{lem:module-finite}, each $B_j$ decomposes into a direct product $\prod_{i=1} ^{r_j} B_{ji}$ of complete local domains, one for each maximal ideal, so that each Artin local factor of $B_j/ (t^m)B_j$ corresponds to a local domain factor of $B_j$.

The above two paragraphs summarize into the following 1-1 correspondences:

$$
\xymatrix{ {\left\{ \begin{matrix} \mbox{Artin local}\\ \mbox{factors of } B_1/ (t^m)B_1 \end{matrix} \right\}}  \ar@{<->}[rr] ^{\mbox{1-1}}\ar@{<->}[d]_{\mbox{1-1}} & & { \left\{ \begin{matrix} \mbox{Artin local}\\ \mbox{factors of } B_2/ (t^m)B_2 \end{matrix} \right\} }  \ar@{<->}[d]^{\mbox{1-1}}  \\
 {\left\{\begin{matrix} \mbox{local factors} \\ \mbox{of } B_1 \end{matrix} \right\} } & & {\left\{\begin{matrix} \mbox{local factors} \\ \mbox{of } B_2 \end{matrix} \right\}}}
$$
In particular, we have the equality of the number of the local factors $r_1= r_2 =:r$. 

For each fixed $j$, after relabeling $B_{ji}$ over $i$ if necessary, the above discussion also shows that we have an isomorphism of $k_m[y_1, \cdots,y_n]$-algebras
\begin{equation}\label{eqn:M=N *1}
\phi_m: B_{1i}/ (t^m) B_{1i} \simeq B_{2i} / (t^m) B_{2i}
\end{equation}
 for all $1 \leq i \leq r$. 
 
 Let $\mathcal{A}_{ji}$ be the coherent $\mathcal{O}_{\square_{\widehat{X}}^n}$-algebra corresponding to $B_{ji}$. Let $f_{ji}: \Spec (\mathcal{A}_{ji}) = \Spec (B_{ji}) \to \Spec ( k[[t]] \{y_1, \cdots, y_n \})$ be the corresponding finite morphism (R. Hartshorne \cite[Exercise II-5.17-(c),(d), p.128]{Hartshorne} or more generally Malgoire-Voisin \cite[1.4.2, 1.4.3]{MV}). Here we have $\mathcal{A}_{ji} \simeq f_{ji*} \mathcal{O}_{\Spec (B_{ji})}$ so that $[\mathcal{A}_{ji}] = f_{ji*} [ \Spec (B_{ji})] = d_{ji} [ f_{ji} ( \Spec (B_{ji}))]$, where $d_{ji}$ is the degree of $f_{ji}$.
  
 Note that $\mathfrak{Z}_{ji}:= f_{ji} (\Spec (B_{ji}))$ is integral and closed in $\Spec (k[[t]] \{ y_1, \cdots, y_n\})$, and its defining ideal is the kernel of the ring homomorphism $f_{ji} ^{\sharp}: k[[t]] \{ y_1, \cdots, y_n \} \to B_{ji}$. Let $C_{ji}= {\rm Im} (f_{ji}^{\sharp})$. By the first isomorphism theorem, this $C_{ji}$ is the coordinate ring of $\mathfrak{Z}_{ji}$. 
 We have the natural factorization 
 $$
 f^{\sharp}_{ji} : k[[t]]\{y_1, \cdots, y_n \} \twoheadrightarrow C_{ji} \hookrightarrow B_{ji}.
 $$
 Going mod $I= (t^m)$, we have a commutative diagram of the induced ring homomorphisms, where $\phi_m$ is the isomorphism of \eqref{eqn:M=N *1} and $\bar{f}_{ji}^{\sharp}$ is the induced homomorphism for $f_{ji} ^{\sharp}$
 $$
 \xymatrix{ k_m [ y_1, \cdots, y_n ]  \ar[rr]^{\bar{f}_{1i}^{\sharp}}  \ar[drr]^{\bar{f}_{2i} ^{\sharp}} & & B_{1i}/ (t^m) B_{1i} \ar[d] ^{\phi_m} _{\simeq} \\
 & & B_{2i} / (t^m) B_{2i}.}
 $$
 As $C_{ji}$ is the image of $f_{ji} ^{\sharp}$, $C_{ji}/ (t^m)C_{ji}$ is also the image of $\bar{f}_{ji} ^{\sharp}.$ In particular, under the identification $\phi_m$, we have an isomorphism of $k_m[y_1, \cdots, y_n]$-algebras
 \begin{equation}\label{eqn:M=N *2}
 C_{1i}/ (t^m) C_{1i} \simeq  C_{2i}/(t^m) C_{2i}.
 \end{equation}
 This is equivalent to having an $\mathcal{O}_{\square_Y^n}$-algebra isomorphism
 $$
 \mathcal{O}_{\mathfrak{Z}_{1i}} \otimes_{\mathcal{O}_{\square_{\widehat{X}} ^n}} \mathcal{O}_{\square_Y^n} \simeq \mathcal{O}_{\mathfrak{Z}_{2i}} \otimes_{\mathcal{O}_{\square_{\widehat{X}} ^n}} \mathcal{O}_{\square_Y^n}.
 $$
 In particular, $[\mathfrak{Z}_{1i}]- [\mathfrak{Z}_{2i}]\in \mathcal{N}^n (\widehat{X}, Y, n)$. 
 
 On the other hand, $B_{ji}$ is a flat $k[[t]]$-algebra (as there is no $t$-torsion, and $k[[t]]$ is a PID), so that the degree $d_{ji}$ can be computed by specializing $f_{ji}^{\sharp}$ mod $t$. We have the field extensions $k \hookrightarrow C_{ji}/(t) \hookrightarrow B_{ji}/(t)$, and $d_{ji}= [ B_{ji}/(t) : C_{ji}/(t)]$. Since $B_{1j}/(t) \simeq B_{2j}/(t)$ and $C_{1j}/(t) \simeq C_{2j}/(t)$ by \eqref{eqn:M=N *1} and \eqref{eqn:M=N *2}, plugging in $1$ into $m$, we deduce that $d_{1i} = d_{2i} =:d_i$.
 
 Thus combining the above discussions, we have
 $$
 [\mathcal{A}_1]- [\mathcal{A}_2] = \sum_{i=1} ^r ( [\mathcal{A}_{1i}]- [\mathcal{A}_{2i}]) = \sum_{i=1} ^r d_i ([ \mathfrak{Z}_{1i}] - [ \mathfrak{Z}_{2i}])  \in \mathcal{N}^n (\widehat{X}, Y, n),
 $$
 thus $\mathcal{M}^n (\widehat{X}, Y, n) \subset \mathcal{N}^n (\widehat{X}, Y, n)$. This proves the equality  \eqref{eqn:M=N}.
\end{proof}

So, for the cycles in the Milnor range over $\Spf (k[[t]])$, we may use the simpler equivalence $\sim_I$, if needed. We do so in \S \ref{sec:strong graph} and beyond.

\subsection{An argument for pairs of cycles}\label{sec:strong graph}

The goal of \S \ref{sec:strong graph} is to prove the following:

\begin{prop}\label{prop:strong graph}
Let $m \geq 2$.
Let $\widehat{X}= \Spf (k[[t]])$ and $I= (t^m) \subset k[[t]]$.
Let $\mathfrak{Z}_1$ and $ \mathfrak{Z}_2 \in z^n  (\widehat{X}, n)$ be integral cycles such that $\mathfrak{Z}_1 \sim_I \mathfrak{Z}_2$. 

Then there exist cycles $E_{\mathfrak{Z}_1}$ and $E_{\mathfrak{Z}_2} \in z^n (\widehat{X}, n+1)$ such that 
$$
\partial E_{\mathfrak{Z}_\ell} = \mathfrak{Z}_\ell- \mathfrak{Z}'_\ell, \ \ \ell = 1, 2
$$
for some positive integer multiples of integral graph cycles, $\mathfrak{Z}'_1$ and $\mathfrak{Z}'_2 \in z^n (\widehat{X}, n)$, satisfying $\mathfrak{Z}_1' \sim_I \mathfrak{Z}_2'$.
\end{prop}

Once proven, we will have two applications. The first is the surjectivity of the graph homomorphism \eqref{eqn:Milnor0}, to be discussed in \S \ref{sec:5.1}.
The second pertains to showing that $gr_{k_m}$ is an isomorphism, discussed in \S \ref{sec:appendix 04}.

\bigskip

Coming back to the proof of Proposition \ref{prop:strong graph}, let $\mathfrak{Z} \in z^n (\widehat{X}, n)$ be an integral cycle. By Corollary \ref{cor:no F_1}, this is given by a prime ideal $\mathfrak{p} \subset k[[t]] \{ y_1, \cdots, y_n\}$ of height $n$. For $1 \leq i \leq n$, we let $\mathfrak{p}_i$ be the prime ideal $\mathfrak{p} \cap k[[t]] \{ y_1, \cdots, y_i\}$. This corresponds to the image $\mathfrak{Z}^{(i)}$ of $\mathfrak{Z}$ under the projection $\Spf (k[[t]]) \times \square_k ^n \to \Spf (k[[t]]) \times \square_k ^i$ that sends $(y_1, \cdots, y_n)\mapsto (y_1, \cdots, y_i)$. For $i=0$, we let $\mathfrak{Z}^{(0)} = \Spf (k[[t]])$ for our notational conveniences. We have $\mathfrak{Z}^{(i)} \in z^i  (\widehat{X}, i)$ for each $1 \leq i \leq n$ as seen before, in the Claim in the middle of the proof of Proposition \ref{prop:triangular}.

\begin{defn}
For an integral $\mathfrak{Z} \in z^n (\widehat{X}, n)$, for each pair $0 \leq i < j \leq n$ of indices, the induced morphism $f^{j/i}: \mathfrak{Z}^{(j)} \to \mathfrak{Z}^{(i)}$ is a finite surjective morphism of integral formal schemes. Define $d^{j/i} (\mathfrak{Z}):= \deg (f^{j/i}: \mathfrak{Z}^{(j)} \to \mathfrak{Z}^{(i)})$. 

\end{defn}

Observe the following, whose proof is obvious by definition:

\begin{lem}\label{lem:degree gr}
The following are equivalent:
\begin{enumerate}
\item $\mathfrak{Z}$ is a graph cycle.
\item $d^{n/0} (\mathfrak{Z}) = 1$
\item For each $1 \leq i \leq n$, we have $d^{i/(i-1)} (\mathfrak{Z}) = 1$.
\end{enumerate}
\end{lem}

We will show that any integral cycle $\mathfrak{Z} \in z^n  (\widehat{X}, n)$ can be turned into an integer multiple of an integral graph cycle modulo the boundary of a cycle in $z^n(\widehat{X}, n+1)$, by eventually achieving Lemma \ref{lem:degree gr}-(3). In the process, we will show that each step of this construction respects the mod $I$-equivalence. This requires the following discussions for pairs of cycles.

\bigskip

Let $\mathfrak{Z}_{\ell} \in z^n  (\widehat{X}, n)$ be two integral cycles for $\ell=1,2$ such that $\mathfrak{Z}_1 \sim_I \mathfrak{Z}_2$. For each $\ell$, by Proposition \ref{prop:triangular}, we have a system of polynomials in $y_1, \cdots, y_n$ for $\mathfrak{Z}_{\ell}$ 
\begin{equation}\label{eqn:triangle}
\tuborg  P_{\ell} ^{(1)} (y_1)  \in k[[t]] [ y_1], \\
 P_{\ell} ^{ (2)} (y_1, y_2)  \in k[[t]] [y_1, y_2], \\
\ \ \ \ \ \vdots \\
P_{\ell} ^{ (n)} (y_1, \cdots, y_n)  \in k[[t]] [y_1, \cdots, y_n],\sluttuborg
\end{equation}
such that all the properties in Proposition \ref{prop:triangular} hold.
Note that $\mathfrak{Z}_1 ^{(i)} \sim_{I} \mathfrak{Z}_2 ^{(i)}$ for each $1 \leq i \leq n$ as well. Furthermore we have:

\begin{lem}\label{lem:deg_pair}
Under the above notations and assumptions, let $\bar{P}_{\ell} ^{(1)} := P_{\ell} ^{(1)}$ and $R_{\ell} ^{(0)}:= k[[t]]$, while for $2 \leq i \leq n$, let $\bar{P}_{\ell} ^{(i)}$ be the image of $P_{\ell} ^{(i)} (y_1, \cdots, y_i)$ in $R_{\ell} ^{(i-1)} [ y_i]$, where 
\begin{equation}\label{eqn:deg_pair2}
R_{\ell} ^{(i-1)}: = k[[t]] \{ y_1, \cdots, y_{i-1} \} / (P_{\ell} ^{(1)}, \cdots , P_{\ell} ^{(i-1)}).
\end{equation}
For each $1 \leq i \leq n$ and each $\ell=1,2$, let $d_{\ell} ^i:= \deg_{y_i} \bar{P}_{\ell} ^{(i)}$. Let $(-1)^{d_{\ell} ^i}c_{\ell} ^{(i)} \in R_{\ell} ^{(i-1)}$ be the constant term of $\bar{P}_{\ell} ^{(i)}$. Let $\bar{R}_{\ell} ^{(i-1)} = R_{\ell} ^{(i-1)} / I R_{\ell} ^{(i-1)}$.

Then we have:

\begin{enumerate}
\item  $d_1 ^i = d_2 ^i$.
\item  $\bar{R}_1 ^{(i-1)} = \bar{R}_2 ^{(i-1)}$. 
\item In the common ring $\bar{R}^{(i-1)}$ of the above, we have $c_{1} ^{(i)} \equiv c_2 ^{(i)}$. 
\end{enumerate}
\end{lem}

Note that by Proposition \ref{prop:triangular}-(2), we have $c_{\ell} ^{(i)} \in  (R_{\ell} ^{ (i-1)} )^{\times}$.

\begin{proof}
We prove (1), (2), (3) at the same time. By the properties of Proposition \ref{prop:triangular} for $P_{\ell} ^{(i)}$, we know that $d_{\ell} ^i = \deg_{y_i} \bar{P}_{\ell} ^{(i)} \geq 1$. 
Going further down mod $I$, let $\bar{\bar{P}}_{\ell} ^{(i)}$ be the image of $\bar{P}_{\ell} ^{(i)} (y_i)$ in $\bar{R}_{\ell}^{(i-1)}[y_i]$.

Note that since $\mathfrak{Z}_{1} ^{(i-1)} \sim_I  \mathfrak{Z}_2 ^{(i-1)}$, we have $\bar{R}_1 ^{(i-1)} = \bar{R}_2 ^{(i-1)}$, proving (2). 

 Since $\bar{P}_{\ell} ^{(i)} (y_i)$ is monic in $y_i$ in the ring ${R}_{\ell}^{(i-1)}[y_i]$ by Proposition \ref{prop:triangular}, its further image $\bar{\bar{P}}_{\ell} ^{(i)} (y_i)$ in $\bar{R}^{(i-1)}[y_i]$ is also monic in $y_i$ with the same $y_i$-degree. Thus
\begin{equation}\label{eqn:deg_pair3}
\deg_{y_i} \bar{P}_{\ell} ^{(i)} = \deg_{y_i} \bar{\bar{P}}_{\ell} ^{(i)}.
\end{equation}
Furthermore, $\mathfrak{Z}_1 ^{(i)} \mod I = \mathfrak{Z}_2 ^{ (i)} \mod I$ as closed subschemes of $\Spec (\bar{R} ^{(i-1)} [y_i])$, and they are respectively given by the monic polynomials $\bar{\bar{P}}_{\ell} ^{(i)}[ y_i] \in \bar{R}^{(i-1)}[y_i]$. Hence we have the equality
\begin{equation}\label{eqn:deg_pair4}
\bar{\bar{P}}_1 ^{(i)} = \bar{\bar{P}}_2 ^{(i)} \mbox{ in } \bar{R}^{(i-1)} [ y_i].
\end{equation}
In particular,
\begin{equation}\label{eqn:deg_pair5}
\deg _{y_i} \bar{\bar{P}}_1 ^{(i)} = \deg_{y_i} \bar{\bar{P}}_{2} ^{(i)}.
\end{equation}
Hence, combining \eqref{eqn:deg_pair3} and \eqref{eqn:deg_pair5}, we deduce (1).

On the other hand, from the equality \eqref{eqn:deg_pair4}, we have the equality of the constant terms $(-1)^{d_1^i} c_1 ^{(i)} \equiv (-1)^{d_2 ^i} c_2 ^{(i)}$ in $\bar{R}^{(i-1)}$. Since $d_1 ^i = d_2 ^i$ by (1), we now have $c_1 ^{(i)} \equiv c_2 ^{(i)}$ in $\bar{R}^{(i-1)}$, proving (3).
\end{proof}

\begin{defn}
By Lemma \ref{lem:deg_pair}-(1), $d_1 ^i = \deg_{y_i} \bar{P}_1 ^{ (i)} = \deg_{y_i} \bar{P}_2 ^{ (i)} = d_2 ^i$. We now simply call it $d_i$. 

We remark that by definition, we have $d_i = d^{i/(i-1)} (\mathfrak{Z}_{\ell})$ for $\ell = 1,2$.\qed
\end{defn}

The following procedure will be used repeatedly.

\begin{lem}\label{lem:str gr 1}
Suppose $1 \leq i \leq n$. Let $\ell= 1, 2$. Let $\mathfrak{Z}_1$ and $\mathfrak{Z}_2 \in z^n (\widehat{X}, n)$ be two integral cycles such that $\mathfrak{Z}_1 \sim_I \mathfrak{Z}_2$, and
\begin{enumerate}
\item [$(\star)$] they have the generators as in \eqref{eqn:triangle} satisfying the properties of Proposition \ref{prop:triangular}. 
\end{enumerate}

Suppose $d^{j/(j-1)} (\mathfrak{Z}_\ell) = 1$ for $j > i$. (We allow $i=n$ as well, in which case there is no condition.) 

Then there exist cycles $C_{\mathfrak{Z}_\ell} \in z^n (\widehat{X}, n+1)$ and positive integer multiples of integral cycles, $\mathfrak{Z}'_1$ and $ \mathfrak{Z}'_2 \in z^n (\widehat{X}, n)$, such that for $\ell = 1,2$, we have $\mathfrak{Z}_\ell - \mathfrak{Z}'_\ell = \partial (C_{\mathfrak{Z}_\ell})$, $\mathfrak{Z}_1' \sim_I \mathfrak{Z}_2'$, the above $(\star)$ holds for $\mathfrak{Z}_{\ell} '$, and
$$
 d^{ j/ (j-1)}  (\mathfrak{Z}_\ell ') = \tuborg 
 \mbox{not necessarily }1 &\mbox{ for } j\geq i+1,\\
1 & \mbox{ for } j = i, \\
 d^{ j/ (j-1)} (\mathfrak{Z}_\ell)& \mbox{ for } j < i \sluttuborg
$$
\end{lem}

\begin{remk}\label{remk:algorithm}
Before we prove Lemma \ref{lem:str gr 1}, let us give an illustration on how one can use it repeatedly to transform an integral cycle $\mathfrak{Z}$ into a positive integer multiple of an integral graph cycle, $\mathfrak{Z}'$. For instance, take $n=3$, and suppose that for an integral cycle $\mathfrak{Z}$, we had the ``degree vector" $(d^{3/2}, d^{2/1}, d^{1/0})^t = (3, 2, 2)^t$, where the superscript $t$ denotes the transpose of the matrix. 

The following shows how the ``degree vectors" change when we repeatedly apply Lemma \ref{lem:str gr 1}, and also how to apply it. 

The numbers over the arrows are the indices $i$ for which we apply Lemma \ref{lem:str gr 1}, and the entries $*$ are the numbers not known with certainty:
$$
\begin{matrix}
\mbox{3rd}\to\\
\mbox{2nd}\to\\
\mbox{1st}\to
\end{matrix}
\begin{bmatrix} 3\\ 2\\ 2 \end{bmatrix} \overset{i=3}{\to}
\begin{bmatrix} \textbf{1}\\ 2\\ 2 \end{bmatrix} \overset{i=2}{\to}
\begin{bmatrix} *\\ \textbf{1}\\ 2 \end{bmatrix} \overset{i=3}{\to}
\begin{bmatrix} \textbf{1}\\ 1\\ 2 \end{bmatrix} \overset{i=1}{\to}
\begin{bmatrix} *\\ *\\ \textbf{1} \end{bmatrix} \overset{i=3}{\to}
\begin{bmatrix} \textbf{1}\\ *\\ 1 \end{bmatrix} \overset{i=2}{\to}
\begin{bmatrix} *\\ \textbf{1}\\ 1 \end{bmatrix} \overset{i=3}{\to}
\begin{bmatrix} \textbf{1}\\ 1 \\ 1 \end{bmatrix}.
$$
In general, for $n\geq 1$, if we are given $(d^{n/ (n-1)}, \cdots, d^{1/0}) ^t= (a_n, \cdots, a_1)^t$ with $a_i \in \mathbb{N}$, then after a finite number of applications of Lemma \ref{lem:str gr 1}, we will get to $(1, \cdots, 1)^t$. The upper bound for the number of times this operation is needed to get to $(1, \cdots, 1)^t$ may be computed by solving a recurrence relation. We guess we need at most $1 + 2 + \cdots + 2^{n-1} = 2^n -1$ times. 
\qed
\end{remk}

\begin{proof}[Proof of Lemma \ref{lem:str gr 1}]
This is a bit technical, although its basic philosophical ideas go back to part of the arguments of \cite{Totaro}.

Since $i$ is fixed here, let $R_{\ell}:= k[ \mathfrak{Z}_{\ell} ^{(i-1)}]$ for $\ell=1,2$. They were previously called $R_{\ell} ^{(i-1)}$ in Lemma \ref{lem:deg_pair}. These are complete local domains finite and free over $k[[t]]$ (Lemma \ref{lem:module-finite}). 
Similarly, let $\tilde{R}_{\ell}:= k[  \mathfrak{Z}_{\ell} ^{ (i)}]$. They were previously called $R_{\ell} ^{(i)}$. 
Here, $R_{\ell} \hookrightarrow \tilde{R}_{\ell}$ is the finite extension of rings corresponding to the projection $f_{\ell} ^{ i/(i-1)}:  \mathfrak{Z}_{\ell} ^{ (i)} \to  \mathfrak{Z}_{\ell} ^{ (i-1)}$.

By the given assumption that $d^{ j/ (j-1)} ( \mathfrak{Z}_{\ell}) = 1$ for $j > i$, the images $\bar{P}_{\ell} ^{ (j)}$ of the polynomials $P_{\ell} ^{ (j)} (y_1, \cdots, y_j)$ in $\tilde{R}_{\ell} [ y_{i+1}, \cdots, y_n]$ are given by
\begin{equation}\label{eqn:y_j const}
\bar{P}_{\ell} ^{ (j)} (y_{i+1}, \cdots, y_n) = y_j - c_{\ell} ^{ (j)} \ \mbox{ for } i < j \leq n,
\end{equation}
where we recall that for $j>i$, we have $d_j = 1$ and $(-1)^{d_j} c_{\ell} ^{(j)} = - c_{\ell} ^{(j)}$ is the constant term of $\bar{P}_{\ell} ^{(j)}$.

On the other hand, recall that $(-1)^{d_i} c_{\ell} ^{(i)}$ is the constant term of the polynomial $\bar{P}_{\ell} ^{(i)} \in R_{\ell}[y_i]$, and we have $c_{\ell} ^{(i)} \in R_{\ell}^{\times}$ by Proposition \ref{prop:triangular}-(2). Let $ (y_i,  y_i ')$ be the coordinates of $\square^2$ in $ \Spf (R_{\ell}) \times \square^{2}$, and consider the closed formal subscheme $C_{\ell} \subset \Spf (R_{\ell}) \times \square ^{2}$ given by the polynomial (\emph{cf.} \cite[Lemma 2]{Totaro})
$$
Q_{\mathfrak{Z}_{\ell}}  (y_i, y_i'):= \bar{P}_{\ell} ^{(i)}  (y_i) - (y_i -1)^{d_i-1} (y_i - c_{\ell}^{(i)} ) y_i ' \in R_\ell [ y_i, y_{i}'].
$$

 Its codimension $1$ faces are 
\begin{equation}\label{eqn:**0}
\tuborg
C_{\ell} \cap \{ y_i = 0 \} = \{ (-1)^{d_i} c_{\ell} ^{(i)} ( 1- y_i') = 0 \} = \{ y_i' = 1 \} = \emptyset, \\
C_{\ell} \cap \{ y_i = \infty \} =^{\dagger} \{ 1 - y_i' = 0 \} = \emptyset,\\
C_{\ell} \cap \{ y_i' = 0 \} = \{ \bar{P}_{\ell} ^{(i)} (y_i) = 0 \},\\
C_{\ell} \cap \{ y_i ' = \infty \} = \{ ( y_i -1)^{d_i -1} (y_i - c_{\ell} ^{(i)}) = 0 \} = \{ y_i = c_{\ell} ^{(i)} \},\sluttuborg
\end{equation}
where $\dagger$ uses that the polynomial $\bar{P}_{\ell} ^{(i)}  (y_i) \in R_{\ell} [ y_i]$ is monic in $y_i$. 
Since these faces in \eqref{eqn:**0} are of codimension $\geq 1$ in $\Spf (R_{\ell}) \times \square ^1$, where they intersect no additional face (cf. Lemma \ref{lem:no face}), we see that $C_{\ell}$ intersects no codimension $\geq 2$ face. Thus we have the condition (\textbf{GP}) of Definition \ref{defn:HCG} for $C_{\ell}$. The condition (\textbf{SF}) of Definition \ref{defn:HCG} is immediate. Thus we have $C_{\ell} \in z^1 (\Spf (R_{\ell}), 2)$. Furthermore, from the calculations in \eqref{eqn:**0}, we deduce that
\begin{equation}\label{eqn:Cell}
\partial C_{\ell} = - \{ \bar{P} ^{(i)}_{\ell} (y_i) = 0 \} + [ c_{\ell} ^{(i)}] \in z^1(\Spf (R_{\ell}), 1).
\end{equation}

For $1 \leq j \leq n$, let $\alpha_{\ell} ^{ (j)} $ be the image of $y_j$ in $k[[t]] \{ y_1, \cdots, y_n \} /  I (\mathfrak{Z}_{\ell})$. For $j > i$, we are given that $d^{j/(j-1)} (\mathfrak{Z}_{\ell}) = 1$ so that we have
\begin{equation}\label{eqn:ss0}
k[[t]] \{ y_1, \cdots, y_n\} /  I (\mathfrak{Z}_{\ell}) = \tilde{R}_{\ell}.
\end{equation}
Under this identification, for $j>i$, $\alpha_{\ell} ^{(j)}$ coincides with the previously defined $c_{\ell} ^{ (j)}$ as used in \eqref{eqn:y_j const}. However these may differ for $j \leq i$.

Let $(Y_1, \cdots, Y_{i-1}, Y_i,  Y_i ', Y_{i+1}, \cdots, Y_n )$ be the coordinates of $\square^{n+1}$ in $  \Spf (R_{\ell}) \times  \square^{n+1}$. Define the cycle $\tilde{C}_{\mathfrak{Z}_{\ell}} \in z^n  (\Spf (R_{\ell}), n+1)$ given by the system of polynomials 
$$
\tilde{C}_{\mathfrak{Z}_{\ell}} :
\tuborg
P_{\ell} ^{(1)} (Y_1), \\
\ \ \ \ \vdots\\
P_{\ell} ^{(i-1)} (Y_1, \cdots, Y_{i-1}), \\
Q_{\mathfrak{Z}_{\ell}} (Y_i, Y_i '), \\
P_{\ell} ^{ (i+1)} (Y_1, \cdots, Y_{i+1}),\\
\ \ \ \ \vdots\\
P_{\ell} ^{ (n)} (Y_1, \cdots, Y_n)\sluttuborg
$$
in $R_{\ell} \{ Y_1, \cdots, Y_{i-1}, Y_i, Y_i', Y_{i+1}, \cdots, Y_n\}$.
Taking \eqref{eqn:y_j const} into account, we use the concatenation notations as in
$$ 
\tilde{C}_{\mathfrak{Z}_{\ell}}= \mathfrak{Z}_{\ell} ^{ (i-1)} \boxtimes C_{\ell} \boxtimes  [ c_{\ell} ^{(i+1)}]\boxtimes \cdots \boxtimes [c_{\ell} ^{(n)}].
$$

Define $C_{\mathfrak{Z}_{\ell}} \in z^n (\widehat{X}, n+1)$ to be
$$
C_{\mathfrak{Z}_{\ell}}:= (-1)^{i}\pi_{\ell, *} \left(\tilde{C}_{\mathfrak{Z}_{\ell}}\right),
$$
where $\pi_{\ell}:= f_{\ell} ^{ (i-1)/0}: \Spf (R_{\ell}) \to \widehat{X}$ is the finite surjective morphism, and 
$$
\pi_{\ell,*}: z^n (\Spf (R_{\ell}), *) \to z^n (\widehat{X}, *)
$$
is the finite push-forward as in Lemma \ref{lem:pf}.

We can compute the boundary of $C_{\mathfrak{Z}_{\ell}}$ by looking at the boundary of $\tilde{C}_{\mathfrak{Z}_{\ell}}$. There, note that both  $\mathfrak{Z}_{\ell} ^{ (i-1)}$ and $[c_{\ell} ^{ (j)}]$ for $j>i$ have no proper face (cf. Lemma \ref{lem:no face}). Consequently, combined with \eqref{eqn:Cell}, we can write
\begin{eqnarray}\label{eqn:bigface}
\notag  &  \ \ \ \ \ \  & \partial \tilde{C}_{\mathfrak{Z}_{\ell}} = (-1)^{i-1} \mathfrak{Z}_{\ell} ^{ (i-1)} \boxtimes ( \partial  C_{\ell}) \boxtimes  [ c_{\ell} ^{(i+1)}]\boxtimes \cdots \boxtimes [c_{\ell} ^{(n)}] \\
   & \ \ \ \ \ \ =&   (-1)^{i} \mathfrak{Z}_{\ell} ^{ (i-1)} \boxtimes   \{ \bar{P}_{\ell} ^{ (i)}= 0 \} \boxtimes [ c_{\ell} ^{(i+1)}]\boxtimes \cdots \boxtimes [c_{\ell} ^{(n)}]  \\
\notag   & \ \ \ \  \ \ -& (-1)^i \mathfrak{Z}_{\ell} ^{ (i-1)} \boxtimes [ c_{\ell} ^{(i)}] \boxtimes [ c_{\ell} ^{(i+1)}]\boxtimes \cdots \boxtimes [c_{\ell} ^{(n)}] \in z^n  (\Spf (R_{\ell}), n).
\end{eqnarray}

We analyze the two terms of \eqref{eqn:bigface}. Consider the closed formal subscheme $\tilde{\mathfrak{Z}}_{\ell} \subset  \Spf (\tilde{R}_{\ell} ) \times \square ^n$ given by the system of polynomials in $\tilde{R}_{\ell} \{ Y_1, \cdots, Y_n \} $
$$
 \tilde{\mathfrak{Z}}_{\ell}: \ \ \{ Y_1- \alpha_{\ell} ^{(1)}, \cdots, Y_n - \alpha_{\ell} ^{ (n)} \}.
$$
 Since $\tilde{R}_{\ell}$ is an integral $k$-domain and all $\alpha_{\ell} ^{ (j)} \in \tilde{R}_{\ell}$ for $1 \leq j \leq n$ under the identification \eqref{eqn:ss0}, the formal scheme $\tilde{\mathfrak{Z}}_{\ell}$ is integral. Then by definition, the first term of \eqref{eqn:bigface} without the sign
 $$
 \mathfrak{Z}_{\ell} ^{ (i-1)} \boxtimes   \{ \bar{P}_{\ell} ^{ (i)}= 0 \} \boxtimes [ c_{\ell} ^{(i+1)}]\boxtimes \cdots \boxtimes [c_{\ell} ^{(n)}]
 $$
  is just ${f_{\ell, *}^{i/(i-1)} } [ \tilde{\mathfrak{Z}}_{\ell}]$ because $\alpha_{\ell} ^{(j)} = c_{\ell} ^{(j)}$ for $j > i$. By definition, we have $f_{\ell, *} ^{i/ 0} (\tilde{\mathfrak{Z}}_{\ell}) = \mathfrak{Z}_{\ell}$  as well. Since $f_{\ell} ^{(i-1)/ 0} \circ f_{\ell} ^{i/ (i-1)}= f_{\ell} ^{ i/0}$ and $f_{\ell} ^{ (i-1)/0}= \pi_{\ell}$, we thus have $\pi_{\ell, *} (\mathfrak{Z}_{\ell} ^{ (i-1)} \boxtimes   \{ \bar{P}_{\ell} ^{ (i)}= 0 \} \boxtimes [ c_{\ell} ^{(i+1)}]\boxtimes \cdots \boxtimes [c_{\ell} ^{(n)}] ) = \mathfrak{Z}_{\ell}$.

This time, consider the closed formal subscheme $\tilde{\mathfrak{Z}}_{\ell}' \subset \Spf (\tilde{R}_{\ell}) \times \square^n$ given by the polynomials in $\tilde{R}_{\ell} \{ Y_1, \cdots, Y_n\}$
$$
\tilde{\mathfrak{Z}}_{\ell}': \ \ \{ Y_1 - \alpha_{\ell} ^{ (1)}, \cdots, Y_{i-1} - \alpha_{\ell} ^{ (i-1)}, Y_i - c_{\ell} ^{(i)}, Y_{i+1} - c_{\ell} ^{ (i+1)}, \cdots, Y_{n} - c_{\ell} ^{ (n)} \}.
$$
Here all $\alpha_{\ell} ^{(j)}$ for $j<i$, and $c_{\ell} ^{(j)}$ for $j \geq i$ are in $\tilde{R}_{\ell}$ so that $\tilde{\mathfrak{Z}}_{\ell}'$ is an integral formal scheme. Note that $c_{\ell} ^{(i)} \in R_{\ell}$ ($\subset \tilde{R}_{\ell}$), and it may not be equal to $\alpha_{\ell} ^{(i)}$ in $\tilde{R}_{\ell}$. By definition 
$$
f_{\ell, *} ^{ i/ (i-1)} (\tilde{\mathfrak{Z}}'_{\ell}) = \mathfrak{Z}_{\ell} ^{ (i-1)} \boxtimes [ c_{\ell} ^{(i)}] \boxtimes [ c_{\ell} ^{(i+1)}]\boxtimes \cdots \boxtimes [c_{\ell} ^{(n)}],
$$
 which is the second term of \eqref{eqn:bigface} without the sign, so this is an integral cycle. Hence applying $\pi_{\ell, *} = f_{\ell, *} ^{ (i-1)/0}$, we obtain a cycle $\pi_{\ell, *} (\mathfrak{Z}_{\ell} ^{ (i-1)} \boxtimes [ c_{\ell} ^{(i)}] \boxtimes [ c_{\ell} ^{(i+1)}]\boxtimes \cdots \boxtimes [c_{\ell} ^{(n)}])$, which is a positive integer multiple of an integral cycle.

Since the boundary operators commute with push-forwards (Lemma \ref{lem:pf}), we have
$$
\partial C_{\mathfrak{Z}_{\ell}} = (-1)^{i} \pi_{\ell,*} (\partial \tilde{C}_{\mathfrak{Z}_{\ell}}) = \mathfrak{Z}_{\ell} - \pi_{\ell, *} ( f_{\ell,*} ^{ i/ (i-1)}(\tilde{\mathfrak{Z}}'_{\ell})) = \mathfrak{Z}_{\ell} - f_{\ell, *} ^{ i/0} (\tilde{\mathfrak{Z}}'_{\ell}).
$$
So, letting $\mathfrak{Z}_{\ell}' := f_{\ell, *} ^{ i/0} (\tilde{\mathfrak{Z}}'_{\ell})$, we have $\partial C_{\mathfrak{Z}_{\ell}} = \mathfrak{Z}_{\ell} - \mathfrak{Z}'_{\ell} $.

\medskip

We now claim that this cycle $\mathfrak{Z}_{\ell}'$ satisfies the desired properties. This is a positive integer multiple of an integral cycle. We can still define $d^{j/(j-1)} (\mathfrak{Z}_{\ell}')$, ignoring the integer multiplicity. By definition, for $j < i$, we have $(\mathfrak{Z}'_{\ell}) ^{(j)} = \mathfrak{Z}_{\ell} ^{ (j)}$, thus $d^{ j/ (j-1)} (\mathfrak{Z}'_{\ell}) = d^{ j/ (j-1)} (\mathfrak{Z}_{\ell})$. 

Since $c_{\ell} ^{(i)} \in R_{\ell}$ and $(\mathfrak{Z}_{\ell}')^{ (i)} =  \mathfrak{Z}_{\ell} ^{ (i-1)} \boxtimes [ c_{\ell} ^{(i)}]$, we have $k[ (\mathfrak{Z}'_{\ell})^{(i)}] = k[ (\mathfrak{Z}'_{\ell}) ^{(i-1)}] = R_{\ell}$. So, the projection $ (\mathfrak{Z}'_{\ell}) ^{ (i)} \to (\mathfrak{Z}_{\ell}')^{(i-1)} = \mathfrak{Z}_{\ell} ^{ (i-1)}$ is of degree $1$, i.e. $d^{i/ (i-1)} (\mathfrak{Z}'_{\ell}) = 1$.

On the other hand, we have $c_{\ell} ^{ (j)}  \in \tilde{R}_{\ell} = k[ \mathfrak{Z}_{\ell} ^{ (i)}]$ for $j > i$, so that $d^{(i+1)/i} (\mathfrak{Z}_{\ell}')$ may not be equal to $1$.

Finally, from that $\mathfrak{Z}_1 \sim_{I} \mathfrak{Z}_2$, by Lemma \ref{lem:deg_pair}, we know that $\mathfrak{Z}_1 ^{ (i-1)} \sim_{I} \mathfrak{Z}_2 ^{ (i-1)}$ so that $c_1 ^{(i)} \equiv c_2 ^{(i)}$ in $R_1/ I R_1 = R_2/IR_2$, while we also know that $\mathfrak{Z}_1 ^{ (i)} \sim_{I} \mathfrak{Z}_2 ^{ (i)}$ so that $c_1 ^{(j)} \equiv c_2 ^{(j)}$ in $\tilde{R}_1 / I \tilde{R}_1 = \tilde{R}_2/ I \tilde{R}_2$. Hence $\mathfrak{Z}_1' \sim_I \mathfrak{Z}_2'$.
This finishes the proof of Lemma \ref{lem:str gr 1}.
\end{proof}

\bigskip

\begin{proof}[Proof of Proposition \ref{prop:strong graph}]
We begin to apply Lemma \ref{lem:str gr 1} to a pair of integral cycles $\mathfrak{Z}_{\ell}$ in $z^n (\widehat{X}, n)$ such that $\mathfrak{Z}_1 \sim_I \mathfrak{Z}_2$. By Proposition \ref{prop:triangular}, they satisfy the assumptions of Lemma \ref{lem:str gr 1}. So we can apply it repeatedly.  

As seen in Remark \ref{remk:algorithm}, in finite number of applications of Lemma \ref{lem:str gr 1}, we eventually get to a pair of integral  cycles with the equal multiplicities, whose degree vectors are $(d^{n/(n-1)} , \cdots, d^{1/0})^t =(1, \cdots, 1)^t$. By Lemma \ref{lem:degree gr}, they are positive integer multiples of integral graph cycles in $z^n (\widehat{X}, n)$. 

Take the sums, say $C_{\ell}$, of the cycles in $z^n (\widehat{X}, n+1)$ appearing in repeated applications of Lemma \ref{lem:str gr 1}. Taking their boundaries, we have the cancellations of all the intermediate $\mathfrak{Z}'_{\ell}$'s, and we have $\partial ( C_{\ell}) = \mathfrak{Z}_{\ell} - \mathfrak{Z}_{\ell}'$, where $\mathfrak{Z}_{\ell}$ is the given cycle in $z^n (\widehat{X}, n)$, and the resulting $\mathfrak{Z}_{\ell}'$ is a positive integer multiple of an integral graph cycle, such that $\mathfrak{Z}_1 ' \sim_I \mathfrak{Z}_2 '$. 
This proves Proposition \ref{prop:strong graph}.
\end{proof}

\section{The proofs of the main theorems}\label{sec:5}

\subsection{The isomorphism over $k[[t]]$} \label{sec:5.1}

The first application of Proposition \ref{prop:strong graph} is the following:

\begin{lem}\label{lem:surj k[[t]]}
Let $k$ be a field. Then the graph homomorphism $gr: K_n ^M (k[[t]]) \to \CH^n (\Spf (k[[t]]), n)$ of \eqref{eqn:Milnor0} is surjective.
\end{lem}

\begin{proof}
Proposition \ref{prop:strong graph} is stated for pairs of integral cycles that are mod $I$-equivalent. We apply it to the identical pair $(\mathfrak{Z}, \mathfrak{Z})$ for each integral cycle $\mathfrak{Z} \in z^n (\Spf (k[[t]]), n)$. In particular, for a cycle $E_{\mathfrak{Z}} \in z^n (\Spf (k[[t]]), n+1)$ and a positive integer multiple of a graph cycle, call it $\Gamma$, we have $\partial E_{\mathfrak{Z}} = \mathfrak{Z} - \Gamma$. Since $\Gamma$ belongs to the image of the graph map $gr$ and $\mathfrak{Z} \equiv \Gamma$ in $\CH^n (\Spf (k[[t]]), n)$, this shows that $gr$ is surjective.
\end{proof}

\begin{remk}
One may give another proof of Lemma \ref{lem:surj k[[t]]} based on arguments of B. Totaro \cite[\S 4]{Totaro}. We outline what extra modifications are needed in addition to the arguments of \emph{loc.cit.} 

In \emph{loc.cit.}, when $F$ is a field, to prove that the graph homomorphism $gr: K_n ^M (F) \to \CH^n (\Spec (F), n)$ is an isomorphism,  Totaro used the norm $N : L ^{\times} \to F ^{\times}$, where $L$ is the residue field of a given point and $F  \subset L$ is a finite extension of fields. 

Specializing to the case $F= k\llp t\rlp$, what we need is that this norm operation for fields is compatible with the finiteness of our cycles over $k[[t]]$. Indeed if a finite ring extension $k[[t]] \to B$ describes an integral cycle $\mathfrak{Z} \in z^n (\Spf (k[[t]]), n)$, then $B$ is a free $k[[t]]$-module and there is the norm map $N: B^{\times} \to k[[t]]^{\times}$ by Corollary \ref{cor:finite free summ} and Definition \ref{defn:free norm}. This is compatible with the norm $N: L ^{\times} \to k\llp t\rlp ^{\times}$, when $L = {\rm Frac}(B)$, in that the following diagram commutes:
\begin{equation}\label{eqn:the norms}
\xymatrix{ B^{\times} \ar@{^(->}[r] \ar[d] ^N & L ^{\times} \ar[d] ^N \\
 k[[t]] ^{\times} \ar@{^(->}[r] & k\llp t\rlp ^{\times}.}
\end{equation}

So, the norms of elements of $B^{\times}$ belong to $k[[t]]^{\times}$. Following along the argument of \cite[\S 4]{Totaro}, together with the above, we may prove that the graph map of Lemma \ref{lem:surj k[[t]]} is surjective. We do not attempt to give details as this is redundant. 
 \qed
\end{remk}

\begin{cor}\label{cor:k[[t]] iso}
Let $k$ be a field with $|k| \gg 0$ so that $K_n ^M (k[[t]]) = \widehat{K}_n ^M (k[[t]])$. 

Then the graph map
$$ 
gr:K_n ^M (k[[t]]) \to \CH^n (\Spf (k[[t]]), n)
$$ 
is an isomorphism.
\end{cor}

\begin{proof}
The map $gr$ is surjective by Lemma \ref{lem:surj k[[t]]} and injectivity by Lemma \ref{lem:inj k[[t]]}. 
\end{proof}

We deduce the following formal-rigid geometric analogue of the Gersten conjecture for higher Chow groups (cf. S. Bloch \cite[Theorem (10.1)]{Bloch HC}):

\begin{cor}\label{cor:formal Gersten}
Let $k$ be a field with $|k| \gg 0$. Then the localization map
$$
\CH^n (\Spf (k[[t]]), n) \to \CH^n (\Sp (k\llp t\rlp ), n)
$$
is injective.
\end{cor}

\begin{proof}
This follows from Corollaries \ref{cor:k((t)) iso} and \ref{cor:k[[t]] iso}, and the Gersten conjecture for $K_n ^M$ of $k[[t]]$ (see M. Kerz \cite[Proposition 10, p.181]{Kerz finite}).
\end{proof}

Lemma \ref{lem:surj k[[t]]} for $k[[t]]$ also implies the surjectivity for $k_m$:

\begin{cor}\label{cor:gr k_m surj}
Let $m \geq 2$.
Let $k$ be a field, $k_m:= k[t]/(t^m)$, and $Y:= \Spec (k_m)$.

Then the graph homomorphism $gr_{k_m} : K_n ^M (k_m) \to \BGH^n (k_m, n)$ of \eqref{eqn:Milnor-1} is surjective.
\end{cor}
\begin{proof}
Recall from \eqref{eqn:defn Milnor} that
$$
\BGH^n (k_m, n)= \frac{ z^n (\Spf (k[[t]]), n)}{\partial  (z^n (\Spf (k[[t]]), n+1)) + \mathcal{M}^n (\Spf (k[[t]]), Y, n)}.
$$ 
Thus together with Lemma \ref{lem:no face}, we have the natural surjective map 
$$
\CH^n (\Spf (k[[t]]), n) \to \BGH^n (k_m, n),
$$
 and we deduce the commutative diagram
$$
\xymatrix{
K_n ^M (k[[t]]) \ar[d] \ar[rr]^{ gr \ \ \ \ } & &  \CH^n (\Spf (k[[t]]), n) \ar[d] \\
K_n ^M (k_m) \ar[rr] ^{gr_{k_m}} &&  \BGH^n (k_m, n).}
$$
Since the top horizontal arrow is surjective by Lemma \ref{lem:surj k[[t]]}, so is the bottom arrow.
\end{proof}

\subsection{The isomorphism over $k_m$}\label{sec:appendix 04}

In Corollary \ref{cor:gr k_m surj}, we saw that $gr_{k_m}$ is surjective. We want to show that it is an isomorphism when $|k|\gg 0$. For this, we want to construct its inverse:

\begin{prop}\label{prop:k_m regulator}
Let $m \geq 2$.
Let $k$ be a field with $|k|  \gg 0$ so that $K_n ^M (k[[t]]) = \widehat{K}_n ^M (k[[t]])$. 

Then there is a homomorphism of groups
\begin{equation}\label{eqn:regulator upsilon_m}
\Upsilon_m: \BGH^n (k_m, n) \to K_n ^M (k_m)
\end{equation}
such that the composite $ \Upsilon_m \circ gr_{k_m}: K_n ^M (k_m) \to \BGH^n (k_m, n) \to K_n ^M (k_m)$ is the identity map.
\end{prop}

To construct the homomorphism \eqref{eqn:regulator upsilon_m}, first consider the composite
\begin{equation}\label{eqn:regulator upsilon_m defn}
\Upsilon_m: z^n (\Spf (k[[t]]), n) \twoheadrightarrow \CH^n (\Spf (k[[t]]), n)
\end{equation}
$$ 
\underset{\sim}{\overset{gr ^{-1}}{\to}} K_n ^M (k[[t]]) \twoheadrightarrow K_n ^M (k_m).
$$

Recall that by Lemma \ref{lem:no face} the first arrow of \eqref{eqn:regulator upsilon_m defn} is the quotient map, while the map $gr$ is an isomorphism by Corollary \ref{cor:k[[t]] iso}, so that the inverse $gr^{-1}$ exists. For the graph cycles, we have the following:

\begin{lem}\label{lem:upsilon_m gr} 
Let $m \geq 2$. For $\Gamma_{(a_1, \cdots, a_n)} \in z^n (\Spf (k[[t]]), n)$ with $a_i \in k[[t]]^{\times}$ for $1 \leq i \leq n$, we have
$$
\Upsilon_m (\Gamma_{(a_1, \cdots, a_n)}) = \{ \bar{a}_1, \cdots, \bar{a}_n\} \ \mbox{ in } K_n ^M (k_m),
$$
where $\bar{a}_i$ is the image of $a_i$ under $ k[[t]]^{\times} \to k_m^{\times}$. 
\end{lem}

\begin{proof}
Since $gr (\{a_1, \cdots, a_n \})$ is the class $[ \Gamma_{(a_1, \cdots, a_n)}]$ in $\CH^n (\Spf (k[[t]]), n)$, we have $gr^{-1} ( [ \Gamma_{(a_1, \cdots, a_n)}]) = \{ a_1, \cdots, a_n \}$ in $K_n ^M (k[[t]])$. The lemma follows immediately.
\end{proof}

\bigskip

To show that \eqref{eqn:regulator upsilon_m defn} induces \eqref{eqn:regulator upsilon_m}, we need to check that $\Upsilon_m$ respects the equivalence relations coming from the boundary $\partial (z^n (\Spf (k[[t]]), n+1))$ and the subgroup $\mathcal{M}^n (\Spf (k[[t]]), Y, n)$. We check these two reciprocity results in Lemma \ref{lem:reciprocity} and Corollary \ref{cor:upsilon t^m well}, respectively:

\begin{lem}\label{lem:reciprocity}
The composite 
$$ 
z^n (\Spf (k[[t]]), n+1) \overset{\partial}{\to} z^n (\Spf (k[[t]]), n) \overset{\Upsilon_m}{\to} K_n ^M (k_m)
$$
is $0$.
\end{lem}

\begin{proof}
By definition, the map $\Upsilon_m$ of \eqref{eqn:regulator upsilon_m defn} already factors through the group $\CH^n (\Spf (k[[t]]), n)$. From this we deduce that $\Upsilon_m \circ \partial = 0$.
\end{proof}

 Note that we have:

\begin{lem}\label{lem:graph upsilon_m}
Let $m \geq 2$.
For two integral graph cycles $\Gamma_1 = \Gamma_{(a_1, \cdots, a_n)}$ and $ \Gamma_2 = \Gamma_{(a_1', \cdots, a_n')}$ in $z^n (\Spf (k[[t]]), n)$ with $a_i, a_i ' \in k[[t]]^{\times}$, we have $\Gamma_1 \sim_I \Gamma_2$ if and only if $a_i \equiv a_i' \mod t^m$ for all $1 \leq i \leq n$.
\end{lem}

\begin{cor}\label{cor:upsilon t^m well}
Let $m \geq 2$.
For two cycles $\mathfrak{Z}_1$ and $\mathfrak{Z}_2$ in $z^n (\Spf (k[[t]]), n)$ such that $\mathfrak{Z}_1 \sim_Y \mathfrak{Z}_2$, we have $\Upsilon_m (\mathfrak{Z}_1) = \Upsilon_m (\mathfrak{Z}_2)$.
\end{cor}

\begin{proof}
By Lemma \ref{lem:mod Y=naive Y}, it is enough to show the statement for two integral cycles $\mathfrak{Z}_1$ and $\mathfrak{Z}_2$ such that  $\mathfrak{Z}_1 \sim_I \mathfrak{Z}_2$, i.e. when their mod $I$-reductions are equal. 

By Proposition \ref{prop:strong graph}, there is a pair $\mathfrak{Z}_1', \mathfrak{Z}_2'$ of positive integer multiples of integral graph cycles such that $\mathfrak{Z}_1' \sim_I \mathfrak{Z}_2'$, and $\mathfrak{Z}_{\ell} - \mathfrak{Z}_{\ell}' = \partial E_{\mathfrak{Z}_{\ell}}$ for some cycles $E_{\mathfrak{Z}_{\ell}} \in z^n( \Spf (k[[t]]) , n+1)$ for $\ell = 1,2$. 

Since $\Upsilon_m \circ \partial = 0$ by Lemma \ref{lem:reciprocity}, we deduce $\Upsilon_m (\mathfrak{Z}_{\ell}) - \Upsilon_m (\mathfrak{Z}_{\ell}') = \Upsilon_m (\partial E_{\mathfrak{Z}_{\ell}}) = 0$ for $\ell = 1,2$. However $\Upsilon_m (\mathfrak{Z}_1') = \Upsilon_m (\mathfrak{Z}_2')$ by Lemmas \ref{lem:upsilon_m gr} and \ref{lem:graph upsilon_m}. Hence 
$$
\Upsilon_m (\mathfrak{Z}_1) = \Upsilon_m (\mathfrak{Z}_1 ' ) = \Upsilon_m (\mathfrak{Z}_2') = \Upsilon_m (\mathfrak{Z}_2), 
$$
 as desired.
\end{proof}

\bigskip

\begin{proof}[Proof of Proposition \ref{prop:k_m regulator}]
By Lemma \ref{lem:reciprocity} and Corollary \ref{cor:upsilon t^m well}, the regulator $\Upsilon_m$ kills the boundaries of cycles and respects the mod $Y$-equivalence. Hence $\Upsilon_m$ of \eqref{eqn:regulator upsilon_m defn} is well-defined on $\BGH^n (k_m, n)$, inducing the homomorphism \eqref{eqn:regulator upsilon_m}.

On the other hand, for an integral graph cycle $\Gamma_{(a_1, \cdots, a_n) } \in z^n (\Spf (k[[t]]), n)$ for some $a_1, \cdots, a_n \in k[[t]]^{\times}$, by Lemma \ref{lem:upsilon_m gr}, we have $\Upsilon_m (\Gamma_{(a_1, \cdots, a_n)}) = \{ \bar{a}_1, \cdots, \bar{a}_n \} \in K_n ^M (k_m)$, where $\bar{a}_i$ is the image of $a_i$ in $k_m ^{\times}$. This implies that $\Upsilon_m \circ gr_{k_m}$ is the identity on $K_n ^M (k_m)$.
\end{proof}

The above discussions now summarize into the following. This proves the part (2) of the first main theorem, Theorem \ref{thm:main intro}:

\begin{thm}\label{thm:k_m main}
Let $m \geq 2$.
Let $k$ be a field with $|k| \gg 0$ so that $K_n ^M (k[[t]]) = \widehat{K}_n ^M (k[[t]])$. 

Then the graph homomorphism 
$$
gr_{k_m} : K_n ^M (k_m) \to \BGH^n (k_m, n)
$$
is an isomorphism.
\end{thm}

\begin{proof}
By Corollary \ref{cor:gr k_m surj}, the map $gr_{k_m}$ is surjective. On the other hand, under our assumption on the cardinality $|k|$, by Proposition \ref{prop:k_m regulator}, there is a homomorphism $\Upsilon_m$ such that the composite $\Upsilon_m \circ gr_{k_m}$ is the identity of $K_n ^M (k_m)$. In particular, the map $gr_{k_m}$ is injective. This proves the theorem.
\end{proof}

Recall that by the condition (\textbf{SF}) of Definition \ref{defn:HCG}, we have the Gysin map $ev_{\CH}: \BGH^n (k_{m+1}, n) \to \CH^n (k, n)$ given by the evaluation $t=0$. This is surjective because $gr_k: K_n ^M (k) \to \CH^n (k, n)$ is an isomorphism by \cite{NS} and \cite{Totaro}, and the map $ev_0$ in the following commutative diagram is surjective: 
\begin{equation}\label{eqn:final rel}
\xymatrix{
K_n ^M (k_{m+1}) \ar[d]_{\simeq} ^{gr_{k_{m+1} }} \ar[rr] ^{ev_0} & &  K_n ^M (k) \ar[d]_{\simeq} ^{gr_k} \\
\BGH^n (k_{m+1}, n) \ar[rr] ^{ev_{\CH}} & &  \CH^n (k, n).}
\end{equation}
Taking the kernels of the horizontal maps, we define the relative parts 
\begin{equation}\label{eqn:rel CH}
\tuborg K_n ^M (k_{m+1}, (t)):= \ker (ev_0)\ \ \mbox{and} \\
\BGH^n ((k_{m+1}, (t)), n) := \ker (ev_{\CH}).\sluttuborg
\end{equation}
We deduce the second main theorem, Theorem \ref{thm:main intro 2}:

\begin{thm}\label{thm:rulling}
Let $m \geq 1$.
Let $k$ be a field of characteristic $0$. 

Then we have an isomorphism of the big de Rham-Witt forms of $k$ of Hesselholt-Madsen \cite{HeMa} with the relative higher Chow groups
$$
\mathbb{W}_m \Omega_k ^{n-1} \simeq \BGH^n  ((k_{m+1}, (t)), n).
$$
\end{thm}

\begin{proof}
The diagram \eqref{eqn:final rel} with the isomorphic vertical maps induces the isomorphism of the relative parts in \eqref{eqn:rel CH}, 
$$
K_n ^M (k_{m+1}, (t)) \overset{\simeq}{\to} \BGH^n ((k_{m+1}, (t)), n).
$$ 
Here $K_n ^M ( k_{m+1}, (t)) \simeq \bigoplus_{i=1} ^{m} \Omega_{k/\mathbb{Z}} ^{n-1}$ (Park-\"Unver \cite[Proposition 5.4.2]{PU Milnor}), while $\mathbb{W}_m \Omega_k ^{n-1} \simeq  \bigoplus_{i=1} ^{m} \Omega_{k/\mathbb{Z}} ^{n-1} $ (R\"ulling \cite[Remark 1.12]{R}). This proves the theorem.
\end{proof}

Theorem \ref{thm:rulling} is the counterpart to a result known for additive higher Chow groups of a field by K. R\"ulling \cite{R}. In an earlier version, it was mistakenly claimed for all fields $k$ using \cite{RS}, but Matthew Morrow pointed out that the reasoning had a jump. The characteristic $p>0$ case remains to be checked. The referee has suggested a good way to argue for this remaining case, and it will be discussed in a separate work.

 \bigskip

\noindent\emph{Acknowledgments.} 
JP thanks Fumiharu Kato and Shane Kelly for some remarks during JP's visit to the Tokyo Institute of Technology in Spring 2019, that influenced the direction of the project.
JP thanks Joseph Ayoub, Chang-Yeon Chough, Youngsu Kim, Wansu Kim, Matthew Morrow and the referee for numerous comments that helped improving the article.

JP thanks Spencer Bloch for his teaching long ago that the Laurent fields should be related to the topic. JP thanks Damy for peace of mind at home and Seungmok Yi for helping him in enduring the long time he works on the project.

During this work, JP was supported by the National Research Foundation of Korea (NRF) grant (2018R1A2B6002287) funded by the Korean government (Ministry of Science and ICT).

\end{document}